\documentclass[12pt]{amsart}
\usepackage{}
\usepackage{amsmath}
\usepackage{amsfonts}
\usepackage{amssymb}
\usepackage[all]{xy}           
\usepackage{bm}
\usepackage{bbding}
\usepackage{txfonts}
\usepackage{amscd}
\usepackage{xspace}
\usepackage[shortlabels]{enumitem}
\usepackage{ifpdf}

\ifpdf
  \usepackage[colorlinks,final,backref=page,hyperindex]{hyperref}
\else
  \usepackage[colorlinks,final,backref=page,hyperindex,hypertex]{hyperref}
\fi
\usepackage{tikz}
\usepackage[active]{srcltx}

\makeatletter

\newtheorem{thm}{Theorem}[section]
\newtheorem{lem}[thm]{Lemma}
\newtheorem{cor}[thm]{Corollary}
\newtheorem{pro}[thm]{Proposition}

\theoremstyle{definition}
\newtheorem{rmk}[thm]{Remark}
\newtheorem{defi}[thm]{Definition}

\newcommand{\nc}{\newcommand}
\newcommand{\delete}[1]{}

\nc{\mlabel}[1]{\label{#1}}  
\nc{\mcite}[1]{\cite{#1}}  
\nc{\mref}[1]{\ref{#1}}  
\nc{\mbibitem}[1]{\bibitem{#1}} 

\delete{
\nc{\mlabel}[1]{\label{#1}{\hfill \hspace{1cm}{\bf{{\ }\hfill(#1)}}}}
\nc{\mcite}[1]{\cite{#1}{{\em{{\ }(#1)}}}}  
\nc{\mref}[1]{\ref{#1}{{\em{{\ }(#1)}}}}  
\nc{\mbibitem}[1]{\bibitem[\em #1]{#1}} 
}

\newcommand {\emptycomment}[1]{}

\nc{\oprn}{\theta}
\nc{\Oprn}{\Theta}

\nc{\calo}{\mathcal{O}}
\nc{\oop}{$\mathcal{O}$-operator\xspace}
\nc{\oops}{$\mathcal{O}$-operators\xspace}
\nc{\mrho}{{\bm{\varrho}}}
\nc{\emk}{\mathbf{K}}
\nc{\invlim}{\displaystyle{\lim_{\longleftarrow}}\,}
\nc{\ot}{\otimes}

\newcommand{\lon }{\,\rightarrow\,}
\newcommand{\be }{\begin{equation}}
\newcommand{\ee }{\end{equation}}

\newcommand{\LR}{$\mathsf{Lie}\mathsf{Rep}$~}
\newcommand{\Coder}{ \mathsf{Coder}}
\newcommand{\MC}{$\mathsf{MC}$~}

\newcommand{\g}{\mathfrak g}
\newcommand{\h}{\mathfrak h}

\newcommand{\huaF}{\mathcal{F}}

\newcommand{\huaX}{\mathcal{X}}

\newcommand{\huaD}{\mathcal{D}}

\newcommand{\huaH}{\mathcal{H}}

\newcommand{\huaO}{{\mathcal{O}}}
\newcommand{\huaT}{\mathcal{T}}

\newcommand{\frki}{\mathfrak i}

\newcommand{\frkl}{\mathfrak l}

\newcommand{\frkp}{\mathfrak p}

\newcommand{\frkC}{\mathfrak C}

\newcommand{\frkX}{\mathfrak X}
\newcommand{\frkY}{\mathfrak Y}

\newcommand{\half}{\frac{1}{2}}

\newcommand{\Courant}[1]{\left\llbracket  #1\right\rrbracket }

\newcommand{\Id}{{\rm{Id}}}

\newcommand{\br}[1]{   [ \cdot,    \cdot  ]   }

\newcommand{\dM}{\mathrm{d}}

\newcommand{\Hom}{\mathrm{Hom}}

\newcommand{\Lie}{\mathrm{Lie}}

\newcommand{\R}{\mathsf{R}}

\newcommand{\RB}{\mathrm{RB}}
\newcommand{\TLB}{\mathrm{TLB}}
\newcommand{\CE}{\mathsf{CE}}

\newcommand{\SN}{\mathsf{SN}}

\newcommand{\NR}{\mathsf{NR}}

\newcommand{\gl}{\mathfrak {gl}}

\newcommand{\ad}{\mathrm{ad}}
\newcommand{\pr}{\mathrm{pr}}

\newcommand{\Img}{\mathrm{Im}}

\newcommand{\K}{\mathbf{K}}

\newcommand{\Sym}{\mathsf{S}}
\newcommand{\Ten}{\mathsf{T}}

\nc{\CV}{\mathbf{C}}
\newcommand{\MN}{\mathsf{MN}}

\newcommand{\noproof}{\begin{flushright} \ensuremath{\square}\end{flushright}}
\begin{document}

\title[Deformations and homotopy theory  of relative Rota-Baxter Lie algebras]{Deformations and homotopy theory   of relative Rota-Baxter Lie algebras}

\author{Andrey Lazarev}
\address{Department of Mathematics and Statistics, Lancaster University, Lancaster LA1 4YF, UK}
\email{a.lazarev@lancaster.ac.uk}

\author{Yunhe Sheng}
\address{Department of Mathematics, Jilin University, Changchun 130012, Jilin, China}
\email{shengyh@jlu.edu.cn}

\author{Rong Tang}
\address{Department of Mathematics, Jilin University, Changchun 130012, Jilin, China}
\email{tangrong@jlu.edu.cn}


\begin{abstract}
We determine the \emph{$L_\infty$-algebra} that controls deformations of a relative Rota-Baxter Lie algebra and show that it is an extension of the dg Lie algebra controlling deformations of the underlying \LR pair by the dg Lie algebra controlling deformations of the relative Rota-Baxter operator. Consequently, we define the {\em cohomology} of relative Rota-Baxter Lie algebras and relate it to their infinitesimal deformations. A large class of relative Rota-Baxter Lie algebras is obtained from triangular Lie bialgebras and we construct a map between the corresponding deformation complexes. Next, the notion of a \emph{homotopy} relative Rota-Baxter Lie algebra is introduced. We show that a class of homotopy relative Rota-Baxter Lie algebras is intimately related to \emph{pre-Lie$_\infty$-algebras}.

\end{abstract}


\keywords{Cohomology, deformation, $L_\infty$-algebra, MC element,   Rota-Baxter algebra, $r$-matrix,  triangular Lie bialgebra}

\maketitle

\vspace{-1.1cm}

\tableofcontents

\allowdisplaybreaks

\section{Introduction}\mlabel{sec:intr}

In this paper we initiate the study of  deformations and cohomology of relative Rota-Baxter Lie algebras and their homotopy versions.

\subsection{Rota-Baxter operators}

The concept of Rota-Baxter operators on associative algebras was
introduced   by G. Baxter \cite{Ba} in his study of
fluctuation theory in probability. Recently it has found many
applications, including Connes-Kreimer's~\cite{CK} algebraic
approach to the renormalization in perturbative quantum field
theory. Rota-Baxter operators lead to the splitting of
operads~\cite{BBGN,PBG}, and are closely related to noncommutative
symmetric functions and Hopf algebras \cite{Fard,Gu0-1,Yu-Guo}.  Recently the relationship between Rota-Baxter operators and double Poisson algebras were studied in \cite{Goncharov}.
In the Lie algebra context, a Rota-Baxter operator was introduced independently in the 1980s as the
operator form of the classical Yang-Baxter equation that plays important roles in many
subfields of mathematics and mathematical physics such as integrable
systems and quantum groups \cite{CP,STS}. For further details on
Rota-Baxter operators, see ~\cite{Gub-AMS,Gub}.

To better understand the classical Yang-Baxter equation and
 related integrable systems, the more general notion of an \oop (later also called
a relative Rota-Baxter operator or a generalized Rota-Baxter operator)
on a Lie algebra was introduced by Kupershmidt~\cite{Ku};
this notion can be traced back to Bordemann
\cite{Bor}. Relative Rota-Baxter operators provide solutions of the classical Yang-Baxter equation in the semidirect product Lie algebra and give rise to pre-Lie algebras \cite{Bai}.

\subsection{Deformations}

The concept of a formal deformation of an algebraic structure began with the seminal
work of Gerstenhaber~\cite{Ge0,Ge} for associative
algebras. Nijenhuis and Richardson   extended this study to Lie algebras
~\cite{NR,NR2}. More generally, deformation theory
for algebras over quadratic operads was developed by Balavoine~\cite{Bal}. For more
general operads we refer the reader to \cite{KSo,LV,Ma}, and the references therein.

There is a well known slogan, often attributed to Deligne, Drinfeld and Kontsevich:  {\em every reasonable deformation theory is controlled by a differential graded (dg) Lie algebra, determined up to quasi-isomorphism}. This slogan has been made into a rigorous theorem by Lurie  and Pridham, cf. \cite{Lu,Pr}, and a recent simple treatment in  \cite{GLST}.

It is also meaningful to deform {\em maps} compatible with given algebraic structures. Recently, the deformation theory of morphisms was   developed in \cite{Borisov,Fregier-Zambon-1,Fregier-Zambon-2},  the deformation theory of $\huaO$-operators was   developed in \cite{TBGS-1} and the deformation theory of diagrams of algebras was studied in \cite{Barmeier,Fregier} using the minimal model of operads and the method of derived brackets \cite{Kosmann-Schwarzbach,Ma-0,Vo}.

Sometimes a dg Lie algebra up to quasi-isomorphism  controlling a deformation theory manifests itself naturally as an \emph{$L_\infty$-algebra}. This often happens when one tries to deform several algebraic structures as well as a compatibility relation between them, such as  diagrams of algebras mentioned above.
We will see that this also happens in the study of  deformations  of a relative Rota-Baxter Lie algebra, which consists of a Lie algebra, its representation and a relative Rota-Baxter operator (see Definition \ref{defi:O} below). We apply Voronov's higher derived brackets construction \cite{Vo} to construct the $L_\infty$-algebra that characterizes  relative Rota-Baxter Lie algebras as Maurer-Cartan ($\mathsf{MC}$) elements in it. This leads, by a well-known procedure of twisting, to an $L_\infty$-algebra controlling deformations of relative Rota-Baxter Lie algebras. Moreover, we show that this  $L_\infty$-algebra is an extension of the dg Lie algebra   that controls deformations of \LR pairs (a \LR pair consists of a Lie algebra and a representation) given in \cite{Arnal} by the  dg Lie algebra that controls deformations of relative Rota-Baxter operators given in \cite{TBGS-1}.

\subsection{Cohomology theories}

A classical approach for studying a mathematical structure is associating invariants to it. Prominent among these are cohomological invariants, or simply cohomology, of various types of algebras. Cohomology controls deformations and extension problems of the corresponding algebraic structures. Cohomology theories of various kinds of algebras have been developed and studied in \cite{Ch-Ei,Ge0,Har,Hor}. More recently these classical constructions have been extended to strong homotopy (or infinity) versions of the algebras, cf. for example \cite{Hamilton-Lazarev}.

In the present paper we study the cohomology theory for relative Rota-Baxter Lie algebras. A relative Rota-Baxter Lie algebra consists of a Lie algebra, its representation and an operator on it together with appropriate compatibility conditions. Constructing the corresponding cohomology theory is not straightforward due to the complexity of these data. We solve this problem by constructing a deformation complex for a {\em relative} Rota-Baxter Lie algebra and endowing it with an \emph{$L_\infty$-structure}. Infinitesimal deformations of relative Rota-Baxter Lie algebras are classified by the second cohomology group. Moreover, we show that there is a long exact sequence of cohomology groups linking the cohomology of \LR pairs introduced in \cite{Arnal}, the cohomology of $\huaO$-operators introduced in \cite{TBGS-1} and the cohomology of relative Rota-Baxter Lie algebras.

The above general framework has two important special cases: Rota-Baxter Lie algebras and triangular Lie bialgebras and we introduce the corresponding cohomology theories for these objects.  We also show that infinitesimal deformations of Rota-Baxter Lie algebras and triangular Lie bialgebras are classified by the corresponding second cohomology groups.

\subsection{Homotopy invariant construction of Rota-Baxter Lie algebras}
Homotopy invariant algebraic structures play a prominent role in modern mathematical physics  ~\cite{St19}.
Historically, the first such structure was that of an $A_\infty$-algebra introduced by Stasheff
in his study of based loop spaces~\cite{Sta63}.
Relevant later  developments include the work of   Lada and Stasheff~\cite{LS,stasheff:shla} about
$L_\infty$-algebras in mathematical physics
and the work of Chapoton and Livernet~\cite{CL} about
pre-Lie$_\infty$-algebras. Strong homotopy (or infinity-) versions of a large class of algebraic
structures were studied in the context of
operads in ~\cite{LV, MSS}.

Dotsenko and Khoroshkin studied the homotopy of Rota-Baxter operators on associative algebras in ~\cite{DK}, and noted that ``in general compact formulas
are yet to be found".  For  Rota-Baxter  Lie algebras, one encounters a similarly challenging situation. In this paper, we use the approach of $L_\infty$-algebras and their $\mathsf{MC}$ elements to formulate the notion of a (strong) homotopy version of a relative Rota-Baxter Lie algebra, which consists of an $L_\infty$-algebra, its representation and a homotopy relative Rota-Baxter operator. We show that strict homotopy relative Rota-Baxter operators give rise to pre-Lie$_\infty$-algebras, and conversely the identity map is a strict homotopy relative Rota-Baxter operator on the subadjacent $L_\infty$-algebra of a  pre-Lie$_\infty$-algebra.

\subsection{Outline of the paper}
In Section \ref{ss:mce}, we briefly recall the deformation theory and the cohomology  of \LR pairs and relative Rota-Baxter operators. In Section \ref{sec:deform}, we establish the deformation theory of relative Rota-Baxter Lie algebras. In Section \ref{sec:cohomology}, we introduce the corresponding cohomology theory and explain how it is related to infinitesimal deformations of relative Rota-Baxter Lie algebras in the usual way. In Section \ref{sec:cohomologyRB}, we study the cohomology theory of  Rota-Baxter Lie algebras. In Section \ref{sec:cohomologyTLB}, we explain how the cohomology theories of  triangular Lie bialgebras and of relative Rota-Baxter Lie algebras are related. In Section \ref{sec:homotopy}, we introduce the notion of a homotopy relative Rota-Baxter operator and characterize it as an \MC element in a certain  $L_\infty$-algebra. Finally, we exhibit a close relationship between homotopy relative Rota-Baxter Lie algebras of a certain kind and pre-Lie$_\infty$-algebras.

\subsection{Notation and conventions}

Throughout this paper, we work with a coefficient field $\K$ which is  of characteristic 0, and $\R$ is a pro-Artinian $\K$-algebra, that is a projective limit of local Artinian $\K$-algebras.

A permutation $\sigma\in\mathbb S_n$ is called an {\em $(i,n-i)$-shuffle} if $\sigma(1)<\cdots <\sigma(i)$ and $\sigma(i+1)<\cdots <\sigma(n)$. If $i=0$ or $n$, we assume $\sigma=\Id$. The set of all $(i,n-i)$-shuffles will be denoted by $\mathbb S_{(i,n-i)}$. The notion of an $(i_1,\cdots,i_k)$-shuffle and the set $\mathbb S_{(i_1,\cdots,i_k)}$ are defined analogously.

Let $V=\oplus_{k\in\mathbb Z}V^k$ be a $\mathbb Z$-graded vector space.
We will denote by $\Sym(V)$ the {\em symmetric  algebra}  of $V$. That is,
$
\Sym(V):=\Ten(V)/I,
$
where $\Ten(V)$ is the tensor algebra and $I$ is the $2$-sided ideal of $\Ten(V)$ generated by all homogeneous elements  of the form
$
x\otimes y-(-1)^{xy}y\otimes x.
$
We will write $\Sym(V)=\oplus_{i=0}^{+\infty}\Sym^i (V)$.
Moreover, we denote the reduced symmetric  algebra by $\bar{\Sym}(V):=\oplus_{i=1}^{+\infty}\Sym^{i}(V)$. Denote the product of homogeneous elements $v_1,\cdots,v_n\in V$ in $\Sym^n(V)$ by $v_1\odot\cdots\odot v_n$. The degree of $v_1\odot\cdots\odot v_n$ is by definition the sum of the degrees of $v_i$. For a permutation $\sigma\in\mathbb S_n$ and $v_1,\cdots, v_n\in V$,  the {\em Koszul sign} $\varepsilon(\sigma)=\varepsilon(\sigma;v_1,\cdots,v_n)\in\{-1,1\}$ is defined by
\begin{eqnarray*}
	v_1\odot\cdots\odot v_n=\varepsilon(\sigma;v_1,\cdots,v_n)v_{\sigma(1)}\odot\cdots\odot v_{\sigma(n)}.
\end{eqnarray*}
The {\em desuspension operator} $s^{-1}$ changes the grading of $V$ according to the rule $(s^{-1}V)^i:=V^{i+1}$. The  degree $-1$ map $s^{-1}:V\lon s^{-1}V$ is defined by sending $v\in V$ to its   copy $s^{-1}v\in s^{-1}V$.

A degree $1$ element $\theta\in\g^1$ is called an {\em \MC element}   of a differential graded Lie algebra $( \oplus_{k\in\mathbb Z}\g^k,[\cdot,\cdot],d)$ if it
satisfies the  {\em \MC equation}:
$
d \theta+\half[\theta,\theta]=0.$

\section{Maurer-Cartan characterizations  of \LR pairs and  relative Rota-Baxter operators}\label{ss:mce}

\subsection{Bidegrees and the Nijenhuis-Richardson bracket}
Let $\g$ be a vector space. For all $n\ge 0$, set $C^n(\g,\g):=\Hom(\wedge^{n+1}\g,\g).$
Let $\g_1$ and $\g_2$ be two vector spaces and elements in $\g_1$ will be denoted by $x,y,z, x_i$ and elements in $\g_2$ will be denoted by $u,v,w,v_i$.
For a   multilinear map $f:\wedge^{k}\g_1\otimes\wedge^{l}\g_2\lon\g_1$, we define  $\hat{f}\in C^{k+l-1}\big(\g_1\oplus\g_2,\g_1\oplus\g_2\big)$ by
\begin{eqnarray*}
\hat{f}\big((x_1,v_1),\cdots,(x_{k+l},v_{k+l})\big):=\sum_{\tau\in\mathbb S_{(k,l)}}(-1)^{\tau}\Big(f(x_{\tau(1)},\cdots,x_{\tau(k)},v_{\tau(k+1)},\cdots,v_{\tau(k+l)}),0\Big).
\end{eqnarray*}
Similarly, for $f:\wedge^{k}\g_1\otimes\wedge^{l}\g_2\lon\g_2$, we define   $\hat{f}\in C^{k+l-1}\big(\g_1\oplus\g_2,\g_1\oplus\g_2\big)$ by
\begin{eqnarray*}
\hat{f}\big((x_1,v_1),\cdots,(x_{k+l},v_{k+l})\big):=\sum_{\tau\in\mathbb S_{(k,l)}}(-1)^{\tau}\Big(0,f(x_{\tau(1)},\cdots,x_{\tau(k)},v_{\tau(k+1)},\cdots,v_{\tau(k+l)})\Big).
\end{eqnarray*}
The linear map $\hat{f}$ is called a {\em lift} of $f$.
We define $\g^{k,l}:=\wedge^k\g_1\otimes\wedge^{l}\g_2$.
The vector space $\wedge^{n}(\g_1\oplus\g_2)$ is isomorphic to the direct sum of $\g^{k,l},~k+l=n$.

\begin{defi}
A linear map $f\in \Hom\big(\wedge^{k+l+1}(\g_1\oplus\g_2),\g_1\oplus\g_2\big)$ has a {\em bidegree} $k|l$, which is denoted by $||f||=k|l$,   if   $f$ satisfies the following two conditions:
\begin{itemize}
\item[\rm(i)] If $X\in\g^{k+1,l}$, then $f(X)\in\g_1$ and if $X\in\g^{k,l+1}$, then $f(X)\in\g_2;$
\item[\rm(ii)]
  In all the other cases $f(X)=0.$
\end{itemize}
We denote the set of homogeneous linear maps of bidegree $k|l$ by $C^{k|l}(\g_1\oplus\g_2,\g_1\oplus\g_2)$.
\end{defi}

It is clear that this gives a well-defined bigrading on the vector space $\Hom\big(\wedge^{k+l+1}(\g_1\oplus\g_2),\g_1\oplus\g_2\big)$.
We have $k+l\ge0,~k,l\ge-1$ because $k+l+1\ge1$ and $k+1,~l+1\ge0$.

Let $\g$ be a vector space. We consider the graded vector space
$
C^*(\g,\g)=\oplus_{n=0}^{+\infty}C^{n}(\mathfrak{g},\mathfrak{g})=\oplus_{n=0}^{+\infty}\Hom(\wedge^{n+1}\g,\g).
$
Then $C^*(\g,\g)$ equipped with the {\em Nijenhuis-Richardson bracket}
\begin{eqnarray}\label{NR-bracket}
[P,Q]_{\NR}=P\bar{\circ}Q-(-1)^{pq}Q\bar{\circ} P,\,\,\,\,\forall P\in C^{p}(\mathfrak{g},\mathfrak{g}),
Q\in C^{q}(\mathfrak{g},\mathfrak{g}),
\end{eqnarray}
is a graded Lie algebra, where  $P\bar{\circ}Q\in C^{p+q}(\mathfrak{g},\mathfrak{g})$ is defined by
\begin{eqnarray}\label{NR-bracket-com}
(P\bar{\circ}Q)(x_{1},\cdots,x_{p+q+1})=\sum_{\sigma\in \mathbb S_{(q+1,p)}}(-1)^{\sigma}
P(Q(x_{\sigma(1)},\cdots,x_{\sigma(q+1)}),
x_{\sigma(q+2)},\cdots,x_{\sigma(p+q+1)}).
\end{eqnarray}

\begin{rmk}\label{NR-coder}
In fact, the Nijenhuis-Richardson bracket is the commutator of  coderivations on the cofree conilpotent cocommutative coalgebra $\bar{\Sym}^c(s^{-1}\g)$. See \cite{NR,Stasheff} for more details.
\end{rmk}

The following lemmas are very important in our later study.

\begin{lem}\label{important-lemma-2}
The Nijenhuis-Richardson bracket on $C^*(\g_1\oplus\g_2,\g_1\oplus\g_2)$ is compatible with the  bigrading. More precisely, if $||f||=k_f|l_f$, $||g||=k_g|l_g$, then $[f,g]_{\NR}$ has bidegree $(k_f+k_g)|(l_f+l_g).$
\end{lem}
\begin{proof}
It follows from direct computation.
\end{proof}

 \begin{rmk}
  In our later study, the subspaces $C^{k|0}(\g_1\oplus\g_2,\g_1\oplus\g_2)$ and $C^{-1|l}(\g_1\oplus\g_2,\g_1\oplus\g_2)$ will be frequently used. By the above lift map, we have the following isomorphisms:
  \begin{eqnarray}
   \label{rep-cochain-1} C^{k|0}(\g_1\oplus\g_2,\g_1\oplus\g_2)&\cong & \Hom(\wedge^{k+1}\g_1,\g_1)\oplus \Hom(\wedge^{k}\g_1\otimes \g_2,\g_2),\\
    \label{rep-cochain-2}C^{-1|l}(\g_1\oplus\g_2,\g_1\oplus\g_2)&\cong & \Hom(\wedge^{l}\g_2,\g_1).
  \end{eqnarray}
\end{rmk}

\begin{lem}\label{Zero-condition-2}
If $||f||=(-1)|k$ and $||g||=(-1)|l$, then $[f,g]_{\NR}=0.$ Consequently, $\oplus_{l=1}^{+\infty}  C^{-1|l}(\g_1\oplus\g_2,\g_1\oplus\g_2)$ is an abelian subalgebra of the graded Lie algebra $(C^*(\g_1\oplus \g_2,\g_1\oplus\g_2),[\cdot,\cdot]_\NR)$
\end{lem}
\begin{proof}
  It follows from Lemma \ref{important-lemma-2}.
\end{proof}

\subsection{\MC characterization, deformations and cohomology of  \LR pairs}

Let $\g$ be a vector space. For $\mu\in C^{1}(\g,\g)=\Hom(\wedge^2\g,\g)$, we have
\begin{eqnarray*}
[\mu,\mu]_{\NR}(x,y,z)=2(\mu\bar{\circ}\mu)(x,y,z)
                      =2\Big(\mu(\mu(x,y),z)+\mu(\mu(y,z),x)+\mu(\mu(z,x),y)\Big).
\end{eqnarray*}
Thus, $\mu$ defines a Lie algebra structure on $\g$ if and only if $[\mu,\mu]_{\NR}=0$.

Define the set of $0$-cochains $\frkC^0_\Lie(\g;\g)$ to be $0$, and define the set of $n$-cochains $\frkC^n_\Lie(\g;\g)$ to be
$$
\frkC^n_\Lie(\g;\g):=\Hom(\wedge^n\g,\g)=C^{n-1}(\g,\g),\quad n\geq 1.
$$
The {\em Chevalley-Eilenberg coboundary operator} $\dM_\CE$   of the Lie algebra $\g$ with coefficients in the adjoint representation  is defined by
\begin{eqnarray}\label{CE-cohomology-diff}
\dM_\CE f=(-1)^{n-1}[\mu,f]_{\NR},\quad \forall f\in \frkC^n_\Lie(\g;\g).
\end{eqnarray}
The resulting cohomology is denoted by $\huaH^*_{\Lie}(\g;\g)$.

\begin{defi}
  A {\em \LR pair} consists of a Lie algebra  $(\g,[\cdot,\cdot]_\g)$  and a representation $\rho:\g\longrightarrow\gl(V)$   of $\g$ on a vector space $V$.
\end{defi}

Usually we will also use $\mu$ to indicate the Lie bracket $[\cdot,\cdot]_\g$, and denote a \LR pair by $(\g,\mu;\rho)$.

Note that $\mu+\rho\in C^{1|0}(\g\oplus V,\g\oplus V)$. Moreover, the fact that $\mu$ is a Lie bracket and $\rho$ is a representation is equivalent to that
$$
[\mu+\rho,\mu+\rho]_\NR=0.
$$
Next, the following result holds:
\begin{pro}{\rm (\cite{Arnal})}\label{deformation-Lie-and-rep}
Let $\g$ and $V$ be two vector spaces. Then $\big(\oplus_{k=0}^{+\infty} C^{k|0}(\g\oplus V,\g\oplus V),[\cdot,\cdot]_{\NR}\big)$ is a graded Lie algebra. Its \MC elements are precisely \LR pairs.{\noproof}
\end{pro}

Let $(\g,\mu;\rho)$ be a \LR pair. By Proposition~\ref{deformation-Lie-and-rep}, $\pi=\mu+\rho$ is an \MC element of the graded Lie algebra $\big(\oplus_{k=0}^{+\infty} C^{k|0}(\g\oplus V,\g\oplus V),[\cdot,\cdot]_{\NR}\big)$.  It follows from the graded Jacobi identity that
$d_\pi:=[\pi,\cdot]_{\NR}$
 is a graded derivation of the graded Lie
algebra $\big(\oplus_{k=0}^{+\infty} C^{k|0}(\g\oplus V,\g\oplus V),[\cdot,\cdot]_{\NR}\big)$ satisfying $d^2_\pi=0$.
  Therefore we have

\begin{thm}{\rm (\cite{Arnal})}\label{thm:deformation-lie-rep}
Let $(\g,\mu;\rho)$ be a \LR pair.  Then $\big(\oplus_{k=0}^{+\infty} C^{k|0}(\g\oplus V,\g\oplus V),[\cdot,\cdot]_{\NR},d_\pi\big)$ is a dg Lie algebra.

Furthermore,   $(\g,\mu+\mu';\rho+\rho')$ is also a \LR pair for $\mu'\in\Hom(\wedge^2\g,\g)$ and $\rho'\in\Hom(\g,\gl(V))$ if and only if $\mu'+\rho'$ is an \MC
element of the dg Lie algebra
$\big(\oplus_{k=0}^{+\infty} C^{k|0}(\g\oplus V,\g\oplus V),[\cdot,\cdot]_{\NR},d_\pi\big)$.{\noproof}
\end{thm}

Let $(\g,\mu;\rho)$ be a \LR pair. Define the set of $0$-cochains $\frkC^0(\g,\rho)$ to be $0$. For $n\geq 1$, we define the set of $n$-cochains $\frkC^n(\g,\rho)$ to be
$$
\frkC^n(\g,\rho):=C^{(n-1)|0}(\g\oplus V,\g\oplus V)=\Hom(\wedge^n\g,\g)\oplus\Hom(\wedge^{n-1}\g\otimes V,V).
$$
Define the coboundary operator $\partial:\frkC^n(\g,\rho)\lon \frkC^{n+1}(\g,\rho)$ by
\begin{equation}\label{defi:coboundary-rep}
  \partial f:=(-1)^{n-1}[ {\mu}+ {\rho}, {f}]_{\NR}.
\end{equation}

By Proposition \ref{deformation-Lie-and-rep}, we deduce that $\partial\circ \partial =0$. Thus we obtain the complex $(\oplus_{n=0}^{+\infty}\frkC^n(\g,\rho),\partial)$.
 \begin{defi}{\rm (\cite{ Arnal})}
   The cohomology of the cochain complex $(\oplus_{n=0}^{+\infty}\frkC^n(\g,\rho),\partial)$ is called  the {\em cohomology of the \LR pair  $(\g,\mu;\rho)$}. The resulting $n$-th cohomology group is denoted by $\huaH^n(\g,\rho)$.
 \end{defi}
Now we give the precise formula for  $\partial$. For any $n$-cochain $f\in \frkC^n(\g,\rho)$, by \eqref{rep-cochain-1}, we will write $f=(f_\g,f_V)$, where $f_\g\in \Hom(\wedge^n\g,\g)$ and $f_V\in\Hom(\wedge^{n-1}\g\otimes V,V)$. Then we have
 \begin{eqnarray}\label{cohomology-algebra-rep}
\partial f=\Big((\partial f)_\g,(\partial f)_V\Big),
\end{eqnarray}
where $(\partial f)_\g=\dM_\CE f_\g$ and $(\partial f)_V$  is given by
\begin{eqnarray}
\nonumber&&(\partial f)_V(x_1,\cdots,x_{n},v)\\
\label{cohomology-algebra-rep-V}&=&\sum_{1\le i<j\le n}(-1)^{i+j}f_V([x_i,x_j]_\g,x_1,\cdots,\hat{x}_i,\cdots,\hat{x}_j,\cdots,x_n,v)+(-1)^{n-1}\rho(f_\g(x_1,\cdots,x_{n}))v\\
\nonumber&&+\sum_{i=1}^{n}(-1)^{i+1}\Big(\rho(x_i)f_V(x_1,\cdots,\hat{x}_i,\cdots,x_n,v)-f_V\big(x_1,\cdots,\hat{x}_i,\cdots,x_n,\rho(x_i)v\big)\Big).
\end{eqnarray}

\subsection{\MC characterization, deformations and cohomologies of relative Rota-Baxter operators}

We now recall the notion of a relative Rota-Baxter operator. Let $(\g,[\cdot,\cdot]_\g)$ be a Lie algebra and $\rho:\g\longrightarrow\gl(V)$   a representation of $\g$ on a vector space $V$.

\begin{defi} \label{defi:O}
\begin{enumerate}
\item[\rm(i)] A linear operator $T:\g\longrightarrow \g$ is called a {\em Rota-Baxter operator } if
\begin{equation} [T(x),T(y)]_\g=T\big([T(x),y]_\g+ [x,T(y)]_\g \big), \quad \forall x, y \in \g.
\label{eq:rbo}
\end{equation}
Moreover, a Lie algebra $(\g,[\cdot,\cdot]_\g)$ with a Rota-Baxter operator $T$ is
called a {\em Rota-Baxter Lie algebra}. We denote it by $(\g,[\cdot,\cdot]_\g,T)$.
\item[\rm(ii)] A {\em relative Rota-Baxter Lie algebra} is a triple $((\g,[\cdot,\cdot]_\g),\rho,T)$, where $(\g,[\cdot,\cdot]_\g)$ is a Lie algebra, $\rho:\g\longrightarrow\gl(V)$ is a representation of $\g$ on a vector space $V$ and $T:V\longrightarrow\g$ is a  {\em relative Rota-Baxter operator}, i.e.
 \begin{equation}
   [Tu,Tv]_\g=T\big(\rho(Tu)(v)-\rho(Tv)(u)\big),\quad\forall u,v\in V.
 \mlabel{eq:defiO}
 \end{equation}

\end{enumerate}
\end{defi}

Note that a Rota-Baxter operator on a Lie algebra is a relative Rota-Baxter operator with respect to the adjoint representation.

  \begin{defi}\label{defi:homoRRB}
  \begin{itemize}
  \item[\rm(i)] Let $(\g,[\cdot,\cdot]_\g,T)$ and $(\g',\{\cdot,\cdot\}_{\g'},T')$ be  Rota-Baxter Lie algebras. A linear map $\phi:\g'\lon\g$ is called a {\em homomorphism} of Rota-Baxter Lie algebras if $\phi$ is a Lie algebra homomorphism and
  $
  \phi\circ T'=T\circ \phi.
  $
 \item[\rm(ii)] Let $((\g,[\cdot,\cdot]_\g),\rho,T)$ and $((\g',\{\cdot,\cdot\}_{\g'}),\rho',T')$ be two relative Rota-Baxter Lie algebras. A {\em homomorphism} from  $((\g',\{\cdot,\cdot\}_{\g'}),\rho',T')$ to $((\g,[\cdot,\cdot]_\g),\rho,T)$ consists of
     a Lie algebra homomorphism  $\phi:\g'\longrightarrow\g$ and a linear map $\varphi:V'\longrightarrow V$ such that
         \begin{eqnarray}
          T\circ \varphi&=&\phi\circ T',\mlabel{defi:isocon1}\\
                \varphi\rho'(x)(u)&=&\rho(\phi(x))(\varphi(u)),\quad\forall x\in\g', u\in V'.\mlabel{defi:isocon2}
      \end{eqnarray}

      In particular, if $\phi$ and $\varphi$ are  invertible,  then $(\phi,\varphi)$ is called an  {\em isomorphism}.
      \end{itemize}
\end{defi}

 Define a skew-symmetric bracket operation on the graded vector space
$\oplus_{k=1}^{+\infty}\Hom(\wedge^{k}V,\g)$
 by
$$
  \Courant{\theta,\phi}:=(-1)^{n-1}[[\mu+\rho,\theta]_{\NR},\phi]_{\NR},\quad\forall \theta\in\Hom(\wedge^nV,\g), \phi\in\Hom(\wedge^mV,\g).
 $$

\begin{pro}{\rm (\cite{TBGS-1})}\label{pro:gla}
  With the above notation, $(\oplus_{k=1}^{+\infty}\Hom(\wedge^{k}V,\g),\Courant{\cdot,\cdot})$ is a graded Lie algebra. Its \MC elements are precisely  relative Rota-Baxter operators on $(\g,[\cdot,\cdot]_\g)$ with respect to the representation $(V;\rho)$.{\noproof}
\end{pro}

Let $T:V\longrightarrow\g$ be a  relative Rota-Baxter operator. By Proposition~\ref{pro:gla}, $T$ is an \MC element of the graded Lie algebra $(\oplus_{k=1}^{+\infty}\Hom(\wedge^{k}V,\g),\Courant{\cdot,\cdot})$.  It follows from graded Jacobi identity that
$d_T:=\Courant{T,\cdot}$
 is a graded derivation on the graded Lie
algebra $(\oplus_{k=1}^{+\infty}\Hom(\wedge^{k}V,\g),\Courant{\cdot,\cdot})$ satisfying $d^2_T=0$.
  Therefore we have

\begin{thm}{\rm (\cite{TBGS-1})}\label{thm:deformation-relative-rb-o}
With the above notation, $(\oplus_{k=1}^{+\infty}\Hom(\wedge^{k}V,\g),\Courant{\cdot,\cdot},d_T)$ is a dg Lie algebra.

Furthermore,
 $T+T'$ is still
a relative Rota-Baxter operator on the Lie algebra $(\g,[\cdot,\cdot]_\g)$ with respect to the
representation $(V;\rho)$ for  $T':V\longrightarrow \g$ if and only if $T'$ is an \MC
element of the dg Lie algebra
$(\oplus_{k=1}^{+\infty}\Hom(\wedge^{k}V,\g),\Courant{\cdot,\cdot},d_T)$.{\noproof}
\end{thm}

Now we define the cohomology governing deformations of a relative Rota-Baxter operator $T:V\lon\g.$ The spaces of $0$-cochains $\frkC^0(T)$ and of $1$-cochains $\frkC^1(T)$ are set to be $0$. For $n\geq2$, define the vector space of $n$-cochains $\frkC^n(T)$ as $\frkC^n(T)=\Hom(\wedge^{n-1}V,\g)$.

Define the coboundary operator $\delta:\frkC^n(T)\lon \frkC^{n+1}(T)$ by
\begin{equation}\label{defi:coboundary-O}
  \delta \theta=(-1)^{n-2}\Courant{T,\theta}=(-1)^{n-2}[[\mu+\rho,T]_{\NR},\theta]_{\NR},\quad \forall \theta\in \Hom(\wedge^{n-1}V,\g).
\end{equation}

By Proposition \ref{pro:gla}, $(\oplus_{n=0}^{+\infty}\frkC^n(T),\delta)$ is a cochain complex.

\begin{defi}{\rm (\cite{TBGS-1})}
  The cohomology of the cochain complex  $(\oplus_{n=0}^{+\infty}\frkC^n(T),\delta)$ is called the {\em cohomology of the relative Rota-Baxter operator} $T:V\lon\g.$ The corresponding $n$-th cohomology group is denoted by $\huaH^n(T)$.
\end{defi}

See \cite{TBGS-1} for explicit formulas of the coboundary operator $\delta.$

\section{Maurer-Cartan characterization and deformations of relative Rota-Baxter Lie algebras}\label{sec:deform}

In this section, we apply Voronov's higher derived brackets to construct the $L_\infty$-algebra that characterizes relative Rota-Baxter Lie algebras as \MC elements. Consequently, we obtain the $L_\infty$-algebra that controls deformations of a relative Rota-Baxter Lie algebra.

\subsection{$L_\infty$-algebras and higher derived brackets}

The notion of an $L_\infty$-algebra was introduced by Stasheff in \cite{stasheff:shla}. See  \cite{LS,LM} for more details.
\begin{defi}
An {\em  $L_\infty$-algebra} is a $\mathbb Z$-graded vector space $\g=\oplus_{k\in\mathbb Z}\g^k$ equipped with a collection $(k\ge 1)$ of linear maps $l_k:\otimes^k\g\lon\g$ of degree $1$ with the property that, for any homogeneous elements $x_1,\cdots,x_n\in \g$, we have
\begin{itemize}\item[\rm(i)]
{\em (graded symmetry)} for every $\sigma\in\mathbb S_{n}$,
\begin{eqnarray*}
l_n(x_{\sigma(1)},\cdots,x_{\sigma(n-1)},x_{\sigma(n)})=\varepsilon(\sigma)l_n(x_1,\cdots,x_{n-1},x_n),
\end{eqnarray*}
\item[\rm(ii)] {\em (generalized Jacobi identity)} for all $n\ge 1$,
\begin{eqnarray*}\label{sh-Lie}
\sum_{i=1}^{n}\sum_{\sigma\in \mathbb S_{(i,n-i)} }\varepsilon(\sigma)l_{n-i+1}(l_i(x_{\sigma(1)},\cdots,x_{\sigma(i)}),x_{\sigma(i+1)},\cdots,x_{\sigma(n)})=0.
\end{eqnarray*}
\end{itemize}
\end{defi}

There is a canonical way to view a differential graded Lie algebra as an $L_\infty$-algebra.

\begin{lem}\label{Quillen-construction}
Let $(\g,[\cdot,\cdot]_\g,d)$ be a dg Lie algebra. Then  $(s^{-1}\g,\{l_i\}_{i=1}^{+\infty})$ is an $L_\infty$-algebra, where $
l_1(s^{-1}x)=s^{-1}d(x),~
l_2(s^{-1}x,s^{-1}y)=(-1)^{x}s^{-1}[x,y]_\g,~
l_k=0,$ for all $k\ge 3,
$
and homogeneous elements $x,y\in \g$.{\noproof}
\end{lem}

\begin{defi}\label{def:weakly}
A {\em weakly filtered  $L_\infty$-algebra} is a pair $(\g,\huaF_{\bullet}\g)$, where $\g$ is an $L_\infty$-algebra and $\huaF_{\bullet}\g$ is a descending
filtration of the graded vector space $\g$ such that $\g=\huaF_1\g\supset\cdots\supset\huaF_n\g\supset\cdots$ and
\begin{itemize}
	\item[\rm(i)]
	 there exists $n\geq 1$ such that for all $k\geq n$ it holds that
$	l_k(\g,\cdots,\g)\subset \huaF_{k}\g,$

	\item[\rm(ii)]
	$\g$ is complete with respect to this filtration, i.e. there is an isomorphism of graded vector spaces
	$
	\g\cong\underleftarrow{\lim}~\g/\huaF_n\g
	$.

\end{itemize}
\end{defi}

\begin{defi}
The set of {\em \MC elements}, denoted by $\mathsf{MC}(\g)$, of a weakly filtered $L_\infty$-algebra $(\g,\huaF_{\bullet}\g)$ is the set of those $\alpha\in \g^0$ satisfying the \MC equation
\begin{eqnarray}\label{MC-equation}
\sum_{k=1}^{+\infty}\frac{1}{k!}l_k(\alpha,\cdots,\alpha)=0.
\end{eqnarray}
\end{defi}

Let $\alpha$ be an \MC element. Define $l_k^{\alpha}:\otimes^k\g\lon\g$  $(k\ge 1)$ by
\begin{eqnarray}
l_k^{\alpha}(x_1,\cdots,x_k)=\sum_{n=0}^{+\infty}\frac{1}{n!}l_{k+n}(\underbrace{\alpha,\cdots,\alpha}_n,x_1,\cdots,x_k).
\end{eqnarray}
\begin{rmk}The condition of being weakly filtered ensures convergence of the series figuring in the definition of \MC elements and \MC twistings above. Note that the notion of a \emph{filtered} $L_\infty$-algebra is due to Dolgushev and Rogers \cite{Dolgushev-Rogers}. For our purposes the weaker notion defined above suffices. 	
\end{rmk}

The following result is essentially contained in \cite[Section 4]{Getzler}; that paper works with a different type of $L_\infty$ algebras than weakly filtered ones, but this does not affect the arguments.

\begin{thm}\label{deformation-mc}
With the above notation, $(\g,\{l_k^{\alpha}\}_{k=1}^{+\infty})$ is a weakly filtered $L_\infty$-algebra, obtained from $\g$ by twisting with the \MC element $\alpha$. Moreover, $\alpha+\alpha'$ is an \MC element of   $(\g,\huaF_{\bullet}\g)$ if and only if $\alpha'$ is an \MC element of the twisted $L_\infty$-algebra $(\g,\{l_k^{\alpha}\}_{k=1}^{+\infty})$.\qed
\end{thm}

One method for constructing explicit $L_\infty$-algebras is given by Voronov's derived brackets \cite{Vo}. Let us recall this construction.

\begin{defi}
A $V$-data consists of a quadruple $(L,\h,P,\Delta)$ where
\begin{itemize}
\item[$\bullet$] $(L,[\cdot,\cdot])$ is a graded Lie algebra,
\item[$\bullet$] $\h$ is an abelian graded Lie subalgebra of $(L,[\cdot,\cdot])$,
\item[$\bullet$] $P:L\lon L$ is a projection, that is $P\circ P=P$, whose image is $\h$ and kernel is a  graded Lie subalgebra of $(L,[\cdot,\cdot])$,
\item[$\bullet$] $\Delta$ is an element in $  \ker(P)^1$ such that $[\Delta,\Delta]=0$.
\end{itemize}
\end{defi}

\begin{thm}{\rm (\cite{Vo})}\label{thm:db}
Let $(L,\h,P,\Delta)$ be a $V$-data. Then $(\h,\{{l_k}\}_{k=1}^{+\infty})$ is an $L_\infty$-algebra where
\begin{eqnarray}\label{V-shla}
l_k(a_1,\cdots,a_k)=P\underbrace{[\cdots[[}_k\Delta,a_1],a_2],\cdots,a_k],\quad\mbox{for homogeneous}~   a_1,\cdots,a_k\in\h.
\end{eqnarray}
 We call $\{{l_k}\}_{k=1}^{+\infty}$ the {\em higher derived brackets} of the $V$-data $(L,\h,P,\Delta)$. \qed
\end{thm}

There is also an $L_\infty$-algebra structure on a bigger space, which is used to study simultaneous deformations of morphisms between Lie algebras in \cite{Barmeier,Fregier-Zambon-1,Fregier-Zambon-2}.

\begin{thm}{\rm (\cite{Vo})}\label{thm:db-big-homotopy-lie-algebra}
Let $(L,\h,P,\Delta)$ be a $V$-data. Then the graded vector space $s^{-1}L\oplus \h$  is an $L_\infty$-algebra where
\begin{eqnarray*}\label{V-shla-big-algebra}
l_1(s^{-1}x,a)&=&(-s^{-1}[\Delta,x],P(x+[\Delta,a])),\\
l_2(s^{-1}x,s^{-1}y)&=&(-1)^xs^{-1}[x,y],\\
l_k(s^{-1}x,a_1,\cdots,a_{k-1})&=&P[\cdots[[x,a_1],a_2]\cdots,a_{k-1}],\quad k\geq 2,\\
l_k(a_1,\cdots,a_{k-1},a_k)&=&P[\cdots[[\Delta,a_1],a_2]\cdots,a_{k}],\quad k\geq 2.
\end{eqnarray*}
Here $a,a_1,\cdots,a_k$ are homogeneous elements of $\h$ and $x,y$ are homogeneous elements of $L$. All the other $L_\infty$-algebra products that are not obtained from the ones written above by permutations of arguments, will vanish.\qed
\end{thm}

\begin{rmk}\label{thm:db-big-homotopy-lie-algebra-small}{\rm (\cite{Fregier-Zambon-1})}
Let $L'$ be a graded Lie subalgebra of $L$ that satisfies $[\Delta,L']\subset L'$. Then $s^{-1}L'\oplus \h$ is an $L_\infty$-subalgebra of the above $L_\infty$-algebra $(s^{-1}L\oplus\h,\{l_k\}_{k=1}^{+\infty})$.
\end{rmk}

\subsection{The $L_\infty$-algebra that controls deformations of relative Rota-Baxter Lie algebras}
Let $\g$ and $V$ be two vector spaces. Then we have a graded Lie algebra $(\oplus_{n=0}^{+\infty}C^{n}(\g\oplus V,\g\oplus V),[\cdot,\cdot]_{\NR})$. This graded Lie algebra gives rise to a V-data, and an $L_\infty$-algebra naturally.
\begin{pro}\label{pro:VdataL}
We have a $V$-data $(L,\h,P,\Delta)$ as follows:
\begin{itemize}
\item[$\bullet$] the graded Lie algebra $(L,[\cdot,\cdot])$ is given by $\big(\oplus_{n=0}^{+\infty}C^{n}(\g\oplus V,\g\oplus V),[\cdot,\cdot]_{\NR}\big)$;
\item[$\bullet$] the abelian graded Lie subalgebra $\h$ is given by
\begin{equation}\label{defi:h}
\h:=\oplus_{n=0}^{+\infty}C^{-1|(n+1)}(\g\oplus V,\g\oplus V)=\oplus_{n=0}^{+\infty}\Hom(\wedge^{n+1}V,\g);
\end{equation}
\item[$\bullet$] $P:L\lon L$ is the projection onto the subspace $\h$;
\item[$\bullet$] $\Delta=0$.
\end{itemize}

Consequently, we obtain an $L_\infty$-algebra $(s^{-1}L\oplus\h,\{l_k\}_{k=1}^{+\infty})$, where $l_i$ are given by
\begin{eqnarray*}
l_1(s^{-1}Q,\theta)&=&P(Q),\\
  l_2(s^{-1}Q,s^{-1}Q')&=&(-1)^Qs^{-1}[Q,Q']_{\NR}, \\
l_k(s^{-1}Q,\theta_1,\cdots,\theta_{k-1})&=&P[\cdots[Q,\theta_1]_{\NR},\cdots,\theta_{k-1}]_{\NR},
\end{eqnarray*}
for homogeneous elements   $\theta,\theta_1,\cdots,\theta_{k-1}\in \h$, homogeneous elements  $Q,Q'\in L$ and all the other possible combinations vanish.
\end{pro}
\begin{proof}
  Note that
  $
  \h=\oplus_{n=0}^{+\infty}C^{-1|(n+1)}(\g\oplus V,\g\oplus V)=\oplus_{n=0}^{+\infty}\Hom(\wedge^{n+1}V,\g).
  $
By Lemma \ref{Zero-condition-2},  we deduce that $\h$ is an abelian subalgebra of $(L,[\cdot,\cdot])$.

   Since $P$ is the projection onto $\h$, it is obvious that $P\circ P=P$. It is also straightforward to see that the kernel  of $P$ is a  graded Lie subalgebra of $(L,[\cdot,\cdot])$. Thus $(L,\h,P,\Delta=0)$ is a V-data.

   The other conclusions follows immediately from Theorem \ref{thm:db-big-homotopy-lie-algebra}.
\end{proof}
By Lemma \ref{important-lemma-2}, we obtain that
\begin{equation}\label{defi:Lprime}
L'=\oplus_{n=0}^{+\infty}C^{n|0}(\g\oplus V,\g\oplus V),\quad\mbox{where}\quad C^{n|0}(\g\oplus V,\g\oplus V)=\Hom(\wedge^{n+1}\g,\g)\oplus\Hom(\wedge^{n}\g\otimes V,V)
 \end{equation}is a graded Lie subalgebra of $\big(\oplus_{n=0}^{+\infty}C^{n}(\g\oplus V,\g\oplus V),[\cdot,\cdot]_{\NR}\big)$.

\begin{cor}\label{cor:Linfty}
 With the above notation,  $(s^{-1}L'\oplus \h,\{l_i\}_{i=1}^{+\infty})$ is an $L_\infty $-algebra, where $l_i$ are given by
 \begin{eqnarray*}
  l_2(s^{-1}Q,s^{-1}Q')&=&(-1)^Qs^{-1}[Q,Q']_{\NR}, \\
l_k(s^{-1}Q,\theta_1,\cdots,\theta_{k-1})&=&P[\cdots[Q,\theta_1]_{\NR},\cdots,\theta_{k-1}]_{\NR},
\end{eqnarray*}
for homogeneous elements   $\theta_1,\cdots,\theta_{k-1}\in \h$, homogeneous elements  $Q,Q'\in L'$, and all the other possible combinations vanish.

Moreover, $(s^{-1}L'\oplus \h,\{l_i\}_{i=1}^{+\infty})$ is weakly filtered with $n=3$ in the sense of Definition \ref{def:weakly} with the filtration given by  $$
\huaF_1=s^{-1}L'\oplus \h,\quad \huaF_2=P[s^{-1}L'\oplus \h,\h]_\NR,\cdots, \huaF_k=P\underbrace{[\cdots[}_ks^{-1}L'\oplus \h,\h]_{\NR},\cdots,\h]_{\NR}, \cdots.
$$
\end{cor}
\begin{proof}
  The stated formulas for the $L_\infty$-structure  follow from Remark \ref{thm:db-big-homotopy-lie-algebra-small} and Proposition \ref{pro:VdataL}. To see that the given filtration satisfies the conditions of Definition \ref{def:weakly} it suffices to note that any element $h\in \h$ can be written as $h=\sum_{i=1}^{+\infty}{h_i}$ where $h_i\in \Hom(\wedge^i V,\g)$ and that the term $h_1:V\to\g$ is
  \emph{nilpotent} (even has square zero) when viewed as an endomorphism of $\g\oplus V$.
\end{proof}

Now we are ready to formulate the main result in this subsection.

\begin{thm}\label{deformation-rota-baxter}
  Let $\g$ and $V$ be two vector spaces,  $\mu\in\Hom(\wedge^2\g,\g),~\rho\in\Hom(\g\otimes V,V)$ and $T\in\Hom(V,\g)$. Then $((\g,\mu),\rho,T)$ is a relative Rota-Baxter Lie algebra if and only if  $(s^{-1}\pi,T)$ is an \MC element of the $L_\infty$-algebra $(s^{-1}L'\oplus \h,\{l_i\}_{i=1}^{+\infty})$ given in Corollary \ref{cor:Linfty}, where $\pi=\mu+\rho\in C^{1|0}(\g\oplus V,\g\oplus V)$.
\end{thm}

\begin{proof}
Since $(s^{-1}L'\oplus \h)^0=s^{-1}\big(\Hom(\g\wedge\g,\g)\oplus\Hom(\g\otimes V,V)\big)\oplus\Hom(V,\g)\subset\huaF_{1}(s^{-1}L'\oplus \h)$,   the   \MC equation is well defined. Let $(s^{-1}\pi,T)$ be an \MC element of $(s^{-1}L'\oplus \h,\{l_i\}_{i=1}^{+\infty})$. By Lemma \ref{important-lemma-2} and Lemma \ref{Zero-condition-2}, we have
\begin{eqnarray*}
||[\pi,T]_{\NR}||=0|1,\quad ||[[\pi,T]_{\NR},T]_{\NR}||=-1|2,\quad [[[\pi,T]_{\NR},T]_{\NR},T]_{\NR}=0.
\end{eqnarray*}
Then, by Corollary \ref{cor:Linfty}, we have
\begin{eqnarray*}
&&\sum_{k=1}^{+\infty}\frac{1}{k!}l_k\Big((s^{-1}\pi,T),\cdots,(s^{-1}\pi,T)\Big)\\
&=&\frac{1}{2!}l_2\Big((s^{-1}\pi,T),(s^{-1}\pi,T)\Big)+\frac{1}{3!}l_3\Big((s^{-1}\pi,T),(s^{-1}\pi,T),(s^{-1}\pi,T)\Big)\\
&=&\Big(-s^{-1}\frac{1}{2}[\pi,\pi]_{\NR},\frac{1}{2}[[\pi,T]_{\NR},T]_{\NR}\Big)\\
&=&(0,0).
\end{eqnarray*}
Thus, we obtain $
 ~[\pi,\pi]_{\NR}=0$ and $
[[\pi,T]_{\NR},T]_{\NR}=0,
$
 which implies that  $(\g,\mu)$ is a Lie algebra, $(V;\rho)$ is its representation and   $T$ is a relative Rota-Baxter operator on  the Lie algebra $(\g,\mu)$ with respect to the representation $(V;\rho)$.
\end{proof}

\begin{rmk}
 Since the axiom defining
a relative Rota-Baxter Lie algebra is not quadratic, it can be anticipated that the deformation complex of a Rota-Baxter Lie algebra is a fully-fledged $L_\infty$-algebra rather than a differential graded Lie algebra.
\end{rmk}

Let $((\g,\mu),\rho,T)$ be a relative Rota-Baxter Lie algebra. Denote by $\pi=\mu+\rho\in C^{1|0}(\g\oplus V,\g\oplus V)$. By Theorem \ref{deformation-rota-baxter}, we obtain that $(s^{-1}\pi,T)$ is an \MC element of the $L_\infty$-algebra  $(s^{-1}L'\oplus \h,\{l_i\}_{i=1}^{+\infty})$ given in Corollary \ref{cor:Linfty}.
 Now  we are ready to give the
  $L_\infty$-algebra that controls deformations of the relative Rota-Baxter Lie algebra.

\begin{thm}\label{thm:Simultaneous-deformation}
With the above notation, we have the twisted $L_\infty$-algebra $\big(s^{-1}L'\oplus \h,\{l_k^{(s^{-1}\pi,T)}\}_{k=1}^{+\infty}\big)$ associated to a relative Rota-Baxter Lie algebra $((\g,\mu),\rho,T)$, where $\pi=\mu+\rho$.

Moreover, for
linear maps $T'\in\Hom(V,\g)$, $\mu'\in\Hom(\wedge^2\g,\g)$ and $\rho'\in\Hom(\g,\gl(V))$, the triple $((\g,\mu+\mu'),\rho+\rho',T+T')$ is again a relative Rota-Baxter Lie algebra if and only if $\big(s^{-1}(\mu'+\rho'),T'\big)$ is an \MC
element of the twisted $L_\infty$-algebra $\big(s^{-1}L'\oplus \h,\{l_k^{(s^{-1}\pi,T)}\}_{k=1}^{+\infty}\big)$.
\end{thm}

\begin{proof}
If $((\g,\mu+\mu'),\rho+\rho',T+T')$ is   a relative Rota-Baxter Lie algebra, then
by Theorem \ref{deformation-rota-baxter}, we deduce that $(s^{-1}(\mu+\mu'+\rho+\rho'),T+T')$ is an \MC
element of the $L_\infty$-algebra given in Corollary \ref{cor:Linfty}. Moreover, by Theorem \ref{deformation-mc}, we obtain that $(s^{-1}(\mu'+\rho'),T')$ is an \MC
element of the $L_\infty$-algebra $\big(s^{-1}L'\oplus \h,\{l_k^{(s^{-1}\pi,T)}\}_{k=1}^{+\infty}\big)$.
\end{proof}

Let $(\g,\mu)$ be a Lie algebra and $(V;\rho)$ a representation of $(\g,\mu)$.
By Theorem \ref{thm:deformation-lie-rep} and Lemma \ref{Quillen-construction}, we have an $L_\infty$-algebra structure on the graded vector space $\oplus_{k=0}^{+\infty}s^{-1}C^{k|0}(\g\oplus V,\g\oplus V)$.

Let $T:V\longrightarrow\g$ be a relative Rota-Baxter operator on a Lie algebra
$(\g,\mu)$ with respect to a representation $(V;\rho)$.
By Theorem \ref{thm:deformation-relative-rb-o} and Lemma \ref{Quillen-construction}, we have an $L_\infty$-algebra structure on the graded vector space $\oplus_{k=1}^{+\infty}\Hom(\wedge^{k}V,\g)$.

The above $L_\infty$-algebras are related as follows.
\begin{thm}\label{L-infinity-ex}
Let $((\g,\mu),\rho,T)$ be a relative Rota-Baxter Lie algebra. Then the $L_\infty$-algebra $\big(s^{-1}L'\oplus \h,\{l_k^{(s^{-1}\pi,T)}\}_{k=1}^{+\infty}\big)$ is a strict extension of the $L_\infty$-algebra $ \oplus_{k=0}^{+\infty}s^{-1}C^{k|0}(\g\oplus V,\g\oplus V) $ by the $L_\infty$-algebra $ \oplus_{k=1}^{+\infty}\Hom(\wedge^{k}V,\g)$, that is, we have the following short exact sequence of $L_\infty$-algebras:
\begin{equation}
0\longrightarrow\oplus_{k=1}^{+\infty}\Hom(\wedge^{k}V,\g)\stackrel{\iota}{\longrightarrow}s^{-1}L'\oplus \h\stackrel{p}{\longrightarrow} \oplus_{k=0}^{+\infty}s^{-1}C^{k|0}(\g\oplus V,\g\oplus V)\longrightarrow 0,
\end{equation}
where $\iota(\theta)=(0,\theta)$ and $p(s^{-1}f,\theta)=s^{-1}f$.
\end{thm}

\begin{proof}For any $(s^{-1}f,\theta)\in (s^{-1}L'\oplus \h)^{n-2}$,
by Lemma \ref{important-lemma-2} and Lemma \ref{Zero-condition-2}, we obtain that
\begin{eqnarray*}
||[\pi,\theta]_{\NR}||=0|(n-1),\quad ||[[\pi,T]_{\NR},\theta]_{\NR}||=-1|n,\quad[[[\pi,T]_{\NR},T]_{\NR},\theta]_{\NR}=0.
\end{eqnarray*}
Moreover, for $1\le k\le n$, we have
$$
||\underbrace{[\cdots[[}_{ k}f,T]_{\NR},T]_{\NR},\cdots,T]_{\NR}||=(n-1-k)|k,
$$
and for $n+1\le k$, we have
$
\underbrace{[\cdots[[}_{ k}f,T]_{\NR},T]_{\NR},\cdots,T]_{\NR}=0.
$
Therefore, we have
\begin{eqnarray}
\nonumber l_1^{(s^{-1}\pi,T)}(s^{-1}f,\theta)&=&\sum_{k=0}^{+\infty}\frac{1}{k!}l_{k+1}\big(\underbrace{(s^{-1}\pi,T),\cdots,(s^{-1}\pi,T)}_k,(s^{-1}f,\theta)\big)\\
                          \nonumber    &=&l_2(s^{-1}\pi,s^{-1}f)+l_3(s^{-1}\pi,T,\theta)+\frac{1}{n!}l_{n+1}(f,\underbrace{T,\cdots,T}_n)\\
                          \label{eq:l1muT}    &=&\big(-s^{-1}[\pi,f]_{\NR},[[\pi,T]_{\NR},\theta]_{\NR}+\frac{1}{n!}\underbrace{[\cdots[[}_nf,T]_{\NR},T]_{\NR},\cdots,T]_{\NR}\big).
\end{eqnarray}

For any $(s^{-1}f_1,\theta_1)\in (s^{-1}L'\oplus \h)^{n_1-2},~(s^{-1}f_2,\theta_2)\in (s^{-1}L'\oplus \h)^{n_2-2}$, we have
\begin{eqnarray*}
||f_1||=(n_1-1)|0,\quad||\theta_1||=-1|(n_1-1),\quad||f_2||=(n_2-1)|0,\quad||\theta_2||=-1|(n_2-1).
\end{eqnarray*}
By Lemma \ref{important-lemma-2},
$
||[[\pi,\theta_1]_{\NR},\theta_1]_{\NR}||=-1|(n_1+n_2-2),~
||[f_1,\theta_2]_{\NR}||=(n_1-2)|(n_2-1),~
||[f_2,\theta_1]_{\NR}||=(n_2-2)|(n_1-1).
$
By Lemma \ref{Zero-condition-2}, for $ 1\le k$, we have
\begin{eqnarray*}
[[[\cdots[[\pi,\underbrace{T]_{\NR},T]_{\NR}\cdots,T}_{k}]_{\NR},\theta_1]_{\NR},\theta_2]_{\NR}=0.
\end{eqnarray*}
By Lemma \ref{important-lemma-2}, for $1\le k\le n_1-1$, we obtain that
\begin{eqnarray*}
||[[\cdots[[f_1,\underbrace{T]_{\NR},T]_{\NR}\cdots,T}_{k}]_{\NR},\theta_2]_{\NR}||=(n_1-k-2)|(n_2+k-1).
\end{eqnarray*}
By Lemma \ref{Zero-condition-2}, for $ n_1\le k$, we have
$
[[\cdots[[f_1,\underbrace{T]_{\NR},T]_{\NR}\cdots,T}_{k}]_{\NR},\theta_2]_{\NR}=0.
$
By Lemma \ref{important-lemma-2}, for $1\le k\le n_2-1$, we obtain that
$
||[[\cdots[[f_2,\underbrace{T]_{\NR},T]_{\NR}\cdots,T}_{k}]_{\NR},\theta_1]_{\NR}||=(n_2-k-2)|(n_1+k-1).
$
By Lemma \ref{Zero-condition-2}, for $ n_2\le k$, we have
$
[[\cdots[[f_2,\underbrace{T]_{\NR},T]_{\NR}\cdots,T}_{k}]_{\NR},\theta_1]_{\NR}=0.
$
Therefore, we have
\begin{eqnarray*}
&&l_2^{(s^{-1}\pi,T)}\big((s^{-1}f_1,\theta_1),(s^{-1}f_2,\theta_2)\big)\\
&=&\sum_{k=0}^{+\infty}\frac{1}{k!}l_{k+2}\big(\underbrace{(s^{-1}\pi,T),\cdots,(s^{-1}\pi,T)}_k,(s^{-1}f_1,\theta_1),(s^{-1}f_2,\theta_2)\big)\\
                              &=&l_2(s^{-1}f_1,s^{-1}f_2)+l_3(s^{-1}\pi,\theta_1,\theta_2)+\frac{1}{(n_1-1)!}l_{n_1+1}(f_1,\underbrace{T,\cdots,T}_{n_1-1},\theta_2)\\
                              &&+(-1)^{n_1n_2}\frac{1}{(n_2-1)!}l_{n_2+1}(f_2,\underbrace{T,\cdots,T}_{n_2-1},\theta_1)\\
                              &=&\Big((-1)^{n_1-1}s^{-1}[f_1,f_2]_{\NR},[[\pi,\theta_1]_{\NR},\theta_2]_{\NR}+\frac{1}{(n_1-1)!}[[\cdots[[f_1,\underbrace{T]_{\NR},T]_{\NR}\cdots,T}_{n_1-1}]_{\NR},\theta_2]_{\NR}\\
                              &&+(-1)^{n_2n_1}\frac{1}{(n_2-1)!}[[\cdots[[f_2,\underbrace{T]_{\NR},T]_{\NR}\cdots,T}_{n_2-1}]_{\NR},\theta_1]_{\NR}\Big).
\end{eqnarray*}
Similarly, for $  m\geq3$, $(s^{-1}f_i,\theta_i)\in (s^{-1}L'\oplus \h)^{n_i-2}$, $1\le i\le m$, we have
\begin{eqnarray*}
&&\nonumber l_m^{(s^{-1}\pi,T)}\big((s^{-1}f_1,\theta_1),\cdots,(s^{-1}f_m,\theta_m)\big)\\
&=&\Big(0,\sum_{i=1}^{m}(-1)^{\alpha}\frac{1}{(n_i+1-m)!}[\cdots[f_i,\underbrace{T]_{\NR},\cdots,T}_{n_i+1-m}]_{\NR},\theta_1]_{\NR},\cdots,\theta_{i-1}]_{\NR},\theta_{i+1}]_{\NR},\cdots,\theta_{m}]_{\NR}\Big),
\end{eqnarray*}
where $\alpha=n_i(n_1+\cdots+n_{i-1})$. Thus, $\iota$ and $p$ are strict morphisms between $L_\infty$-algebras and satisfy $p\circ\iota=0.$
\end{proof}

\section{Cohomology and  infinitesimal deformations of relative Rota-Baxter Lie algebras}\label{sec:cohomology}
In this section, $((\g,\mu),\rho,T)$ is a relative Rota-Baxter Lie algebra, i.e. $\rho:\g\lon\gl(V)$ is a representation of the Lie algebra $(\g,\mu)$ and $T:V\lon\g$ is a relative Rota-Baxter operator. We define the cohomology of relative Rota-Baxter Lie algebras and show that the two-dimensional cohomology groups classify infinitesimal deformations. We also establish a relationship between the cohomology of relative Rota-Baxter Lie algebras and the cohomology of triangular Lie bialgebras.
\subsection{Cohomology of relative Rota-Baxter Lie algebras}\label{sec:cohomologyRRB}

We define the cohomology of a relative Rota-Baxter Lie algebra using the twisted $L_\infty$-algebra given in Theorem \ref{thm:Simultaneous-deformation}.

By Theorem \ref{thm:Simultaneous-deformation}, $\big(s^{-1}L'\oplus \h,\{l_k^{(s^{-1}\pi,T)}\}_{k=1}^{+\infty}\big)$ is an $L_\infty$-algebra, where $\pi=\mu+\rho$, $\h$ and $L'$ are given by \eqref{defi:h} and \eqref{defi:Lprime} respectively. In particular, we have
 \begin{lem}\label{lem:complex}
 $\big(s^{-1}L'\oplus \h,l_1^{(s^{-1}\pi,T)}\big)$ is a complex, i.e.  $l_1^{(s^{-1}\pi,T)}\circ l_1^{(s^{-1}\pi,T)}=0.$
  \end{lem}
\begin{proof}
  Since $\big(s^{-1}L'\oplus \h,\{l_k^{(s^{-1}\pi,T)}\}_{k=1}^{+\infty}\big)$ is an $L_\infty$-algebra, we have $l_1^{(s^{-1}\pi,T)}\circ l_1^{(s^{-1}\pi,T)}=0.$
\end{proof}

Define the set of  $0$-cochains $\frkC^0(\g,\rho,T)$ to be $0$, and define
the set of  $1$-cochains $\frkC^1(\g,\rho,T)$ to be $\gl(\g)\oplus \gl(V)$. For $n\geq 2$, define the space of    $n$-cochains $\frkC^n(\g,\rho,T)$ by
\begin{eqnarray*}
\frkC^n(\g,\rho,T)&:=&\frkC^n(\g,\rho)\oplus \frkC^n(T)=C^{(n-1)|0}(\g\oplus V,\g\oplus V)\oplus C^{-1|(n-1)}(\g\oplus V,\g\oplus V)\\
&=&\Big(\Hom(\wedge^n\g,\g)\oplus \Hom(\wedge^{n-1}\g\otimes V,V)\Big)\oplus\Hom(\wedge^{n-1}V,\g).
\end{eqnarray*}

Define the {\em coboundary operator} $\huaD:\frkC^n(\g,\rho,T)\lon \frkC^{n+1}(\g,\rho,T)$ by
\begin{equation}\label{cohomology-of-RB}
  \huaD(f,\theta)=(-1)^{n-2}\big(-[\pi,f]_{\NR},[[\pi,T]_{\NR},\theta]_{\NR}+\frac{1}{n!}\underbrace{[\cdots[[}_nf,T]_{\NR},T]_{\NR},\cdots,T]_{\NR}\big),
\end{equation}
where $f\in\Hom(\wedge^n\g,\g)\oplus \Hom(\wedge^{n-1}\g\otimes V,V)$ and $\theta\in \Hom(\wedge^{n-1}V,\g).$

\begin{thm}\label{cohomology-of-relative-RB}
  With the above notation,  $(\oplus _{n=0}^{+\infty}\frkC^n(\g,\rho,T),\huaD)$ is a cochain complex, i.e. $\huaD\circ \huaD=0.$
\end{thm}
\begin{proof}
For any $(f,\theta)\in\frkC^n(\g,\rho,T)$, we have $(s^{-1}f,\theta)\in (s^{-1}L'\oplus \h)^{n-2}$. By \eqref{cohomology-of-RB}, we deduce that
$$
\huaD(f,\theta)=(-1)^{n-2} l_1^{(s^{-1}\pi,T)}(s^{-1}f,\theta).
$$
By Lemma \ref{lem:complex}, we obtain that $(\oplus _{n=0}^{+\infty}\frkC^n(\g,\rho,T),\huaD)$ is a cochain complex.
\end{proof}

\begin{defi}
  The cohomology of the cochain complex $(\oplus _{n=0}^{+\infty}\frkC^n(\g,\rho,T),\huaD)$ is called the {\em cohomology  of the relative Rota-Baxter Lie algebra} $((\g,\mu),\rho,T)$. We denote its $n$-th cohomology group by $\huaH^n(\g,\rho,T)$
\end{defi}

Define a linear operator $h_T:\frkC^n(\g,\rho)\lon \frkC^{n+1}(T)$ by
\begin{eqnarray}\label{key-cohomology-T-abstract}
h_Tf:=(-1)^{n-2} \frac{1}{n!}\underbrace{[\cdots[[}_nf,T]_{\NR},T]_{\NR},\cdots,T]_{\NR}.
\end{eqnarray}
By \eqref{cohomology-of-RB} and \eqref{key-cohomology-T-abstract}, the coboundary operator can be written as
\begin{equation}\label{eq:Dexplicit}
\huaD(f,\theta)=(\partial f,\delta \theta+h_Tf),
\end{equation}
where $\partial $ is given by \eqref{defi:coboundary-rep},
and $\delta $ is given by   \eqref{defi:coboundary-O}. More precisely,
\begin{eqnarray}
&& (\delta \theta)(v_1,\cdots,v_{n})\notag\\
&=&\sum_{i=1}^{n}(-1)^{i+1}[Tv_i,\theta(v_1,\cdots,\hat{v}_i,\cdots, v_{n})]_\g+\sum_{i=1}^{n}(-1)^{i+1}T\rho(\theta(v_1,\cdots,\hat{v}_i,\cdots, v_{n}))(v_i)\label{eq:odiff}\\
&&+\sum_{1\le i<j\le n}(-1)^{i+j}\theta(\rho(Tv_i)(v_j)-\rho(Tv_j)(v_i),v_1,\cdots,\hat{v}_i,\cdots,\hat{v}_j,\cdots, v_{n}). \notag
\end{eqnarray}

Now we give the  formulas for $h_T$ in terms of multilinear maps.
\begin{lem}The operator $h_T:\Hom(\wedge^n\g,\g)\oplus \Hom(\wedge^{n-1}\g\otimes V,V)\lon\Hom(\wedge^{n}V,\g)$ is given by
\begin{eqnarray}\label{key-cohomology-T}
&&\nonumber( h_Tf)(v_1,\cdots,v_n)\\
&=&(-1)^{n}f_\g(Tv_1,\cdots,Tv_n)+\sum_{i=1}^{n}(-1)^{i+1}Tf_V\big(Tv_1,\cdots,Tv_{i-1},Tv_{i+1},\cdots,Tv_n,v_i\big),
\end{eqnarray}
where $f=(f_\g,f_V),$ and $f_\g\in \Hom(\wedge^n\g,\g),~f_V\in \Hom(\wedge^{n-1}\g\otimes V,V)$ and $v_1,\cdots,v_n\in V.$

\end{lem}

\begin{proof}
  By Remark \ref{NR-coder}, it is convenient to view the elements of $\oplus_{n=0}^{+\infty}C^{n}(\g\oplus V;\g\oplus V)$ as coderivations of $\bar{\Sym}^c\big(s^{-1}(\g\oplus V)\big)$. The coderivations corresponding to $f$ and $T$ will be denoted by $\bar{f}$ and $\bar{T}$ respectively.
 Then, by induction, we have
\begin{eqnarray*}
&&\underbrace{[\cdots[[}_nf,T]_{\NR},T]_{\NR},\cdots,T]_{\NR}\big((x_1,v_1),\cdots,(x_{n},v_{n})\big)\\
&=&\sum_{i=0}^{n}(-1)^{i}{n\choose i}\big(\underbrace{\bar{T}\circ\cdots\circ\bar{T}}_i\circ (\bar{f}_\g+\bar{f}_V)\circ \underbrace{\bar{T}\cdots\circ \bar{T}}_{n-i}\big)\big((x_1,v_1),\cdots,(x_{n},v_{n})\big)\\
&=&\big(n!f_\g(Tv_1,\cdots,Tv_n),0\big)\\
&&+\Big((-1)n\sum_{k=1}^{n}(-1)^{n-i}(n-1)!Tf_V\big(Tv_1,\cdots,Tv_{i-1},Tv_{i+1},\cdots,Tv_n,v_i\big),0\Big)\\
&=&\big(n!f_\g(Tv_1,\cdots,Tv_n),0\big)+\Big(n!\sum_{k=1}^{n}(-1)^{n-i+1}Tf_V\big(Tv_1,\cdots,Tv_{i-1},Tv_{i+1},\cdots,Tv_n,v_i\big),0\Big),
\end{eqnarray*}
which implies that \eqref{key-cohomology-T} holds.
\end{proof}

The formula of the coboundary operator $\huaD$ can be well-explained by the following diagram:
 \[
\small{ \xymatrix{
\cdots
\longrightarrow \frkC^n(\g,\rho)\ar[dr]^{h_T} \ar[r]^{\qquad\partial} & \frkC^{n+1}(\g,\rho) \ar[dr]^{h_T} \ar[r]^{\partial\qquad}  & \frkC^{n+2}(\g,\rho)\longrightarrow\cdots  \\
\cdots\longrightarrow \frkC^n(T) \ar[r]^{\qquad\delta} &\frkC^{n+1}(T)\ar[r]^{\delta\qquad}&\frkC^{n+2}(T)\longrightarrow \cdots.}
}
\]
\begin{thm}\label{cohomology-exact}
Let $((\g,\mu),\rho,T)$ be a relative Rota-Baxter Lie algebra. Then there is a short exact sequence of the  cochain complexes:
$$
0\longrightarrow(\oplus_{n=0}^{+\infty}\frkC^n(T),\delta)\stackrel{\iota}{\longrightarrow}(\oplus _{n=0}^{+\infty}\frkC^n(\g,\rho,T),\huaD)\stackrel{p}{\longrightarrow} (\oplus_{n=0}^{+\infty}\frkC^n(\g,\rho),\partial)\longrightarrow 0,
$$
where $\iota$ and $p$ are the inclusion map and the projection map.

Consequently, there is a long exact sequence of the  cohomology groups:
$$
\cdots\longrightarrow\huaH^n(T)\stackrel{\huaH^n(\iota)}{\longrightarrow}\huaH^n(\g,\rho,T)\stackrel{\huaH^n(p)}{\longrightarrow} \huaH^n(\g,\rho)\stackrel{c^n}\longrightarrow \huaH^{n+1}(T)\longrightarrow\cdots,
$$
where the connecting map $c^n$ is defined by
$
c^n([\alpha])=[h_T\alpha],$  for all $[\alpha]\in \huaH^n(\g,\rho).$
\end{thm}
\begin{proof}
 By  \eqref{eq:Dexplicit}, we have the short exact sequence  of chain complexes which induces a long exact sequence of cohomology groups.   Also by \eqref{eq:Dexplicit},   $c^n$ is given by
$
c^n([\alpha])=[h_T\alpha].$
\end{proof}

\subsection{Infinitesimal deformations of relative Rota-Baxter Lie algebras}\label{sec:defRRB}

In this subsection, we introduce the notion of $\R$-deformations of relative Rota-Baxter Lie algebras, where $\R$ is a local pro-Artinian $\K$-algebra. Since $\R$ is the projective limit of local Artinian $\K$-algebras, $\R$ is equipped with an augmentation $\epsilon:\R\lon \K$. See \cite{Doubek-Markl-Zima,KSo} for more details about $\R$-deformation theory of algebraic structures. Then we restrict our study to infinitesimal deformations, i.e. $\R=\K[t]/(t^{2})$,    using the cohomology theory introduced in Section \ref{sec:cohomologyRRB}.

Replacing the $\K$-vector spaces and $\K$-linear maps by $\R$-modules and $\R$-linear maps in Definition \ref{defi:O} and Definition \ref{defi:homoRRB}, it is straightforward to obtain the definitions of $\R$-relative Rota-Baxter Lie algebras and homomorphisms between them.

 Any relative Rota-Baxter Lie algebra $((\g,[\cdot,\cdot]_\g),\rho,T)$ can be viewed as an $\R$-relative Rota-Baxter Lie algebra    with the help of the augmentation map $\epsilon$. More precisely, the $\R$-module structure on $\g$ and $V$ are given by
 $$
 r\cdot x:=\epsilon(r)x,\quad r\cdot u:=\epsilon(r)u,\quad \forall r\in\R,~x\in\g,~u\in V.
 $$

\begin{defi}
An {\em $\R$-deformation} of a relative Rota-Baxter Lie algebra $((\g,[\cdot,\cdot]_\g),\rho,T)$ consists of an $\R$-Lie algebra structure $[\cdot,\cdot]_\R$ on the tensor product $\R\otimes_\K\g$, an $\R$-Lie algebra  homomorphism $\rho_\R:\R\otimes_\K\g\lon\gl_\R(\R\otimes_\K V)$ and  an $\R$-linear map $T_\R:\R\otimes_\K V\lon\R\otimes_\K\g$, which is a relative Rota-Baxter operator
such that $(\epsilon\otimes_\K\Id_\g,\epsilon\otimes_\K\Id_V)$ is an $\R$-relative Rota-Baxter Lie algebra homomorphism from $((\R\otimes_\K\g,[\cdot,\cdot]_\R),\rho_\R,T_\R)$ to $((\g,[\cdot,\cdot]_\g),\rho,T)$.
\end{defi}
Thereafter, we denote an $\R$-deformation of  $((\g,[\cdot,\cdot]_\g),\rho,T)$ by a triple $((\R\otimes_\K\g,[\cdot,\cdot]_\R),\rho_\R,T_\R)$.
Next we discuss   equivalences between   $\R$-deformations.
\begin{defi}
Let $((\R\otimes_\K\g,[\cdot,\cdot]_\R),\rho_\R,T_\R)$ and $((\R\otimes_\K\g,[\cdot,\cdot]'_\R),\rho'_\R,T'_\R)$ be two $\R$-deformations of a relative Rota-Baxter Lie algebra $((\g,[\cdot,\cdot]_\g),\rho,T)$. We call them {\em equivalent} if there exists an $\R$-relative Rota-Baxter Lie algebra isomorphism $(\phi,\varphi):((\R\otimes_\K\g,[\cdot,\cdot]'_\R),\rho'_\R,T'_\R)\lon((\R\otimes_\K\g,[\cdot,\cdot]_\R),\rho_\R,T_\R)$
such that
\begin{eqnarray}\label{equivalent-deformation-2}
(\epsilon\otimes_\K\Id_\g,\epsilon\otimes_\K\Id_V)=(\epsilon\otimes_\K\Id_\g,\epsilon\otimes_\K\Id_V)\circ(\phi,\varphi).
\end{eqnarray}
\end{defi}

\begin{defi}
   A $\K[t]/(t^{2})$-deformation of the relative Rota-Baxter Lie algebra $((\g,[\cdot,\cdot]_\g),\rho,T)$ is call an  {\em infinitesimal deformation}.
\end{defi}
Let $\R=\K[t]/(t^{2})$ and $((\R\otimes_\K\g,[\cdot,\cdot]_\R),\rho_\R,T_\R)$ be an infinitesimal deformation of $((\g,[\cdot,\cdot]_\g),\rho,T)$. Since $((\R\otimes_\K\g,[\cdot,\cdot]_\R),\rho_\R,T_\R)$ is an $\R$-relative Rota-Baxter Lie algebra, there exist $\omega_0,~\omega_1\in\Hom(\g\wedge\g,\g)$,  $\varrho_0,~\varrho_1\in\gl(V)$ and $\huaT_0,~\huaT_1\in\Hom(V,\g)$ such that
\begin{eqnarray}\label{infinitesimal-deformation}
[\cdot,\cdot]_\R=\omega_0+t\omega_1,\quad \rho_\R=\varrho_0+t\varrho_1,\quad T_\R=\huaT_0+t\huaT_1.
\end{eqnarray}
 Since $(\epsilon\otimes_\K\Id_\g,\epsilon\otimes_\K\Id_V)$ is an $\R$-relative Rota-Baxter Lie algebra homomorphism from $((\R\otimes_\K\g,[\cdot,\cdot]_\R),\rho_\R,T_\R)$ to $((\g,[\cdot,\cdot]_\g),\rho,T)$, we deduce that $$\omega_0=[\cdot,\cdot]_\g,\quad \varrho_0=\rho,\quad \huaT_0=T.$$ Therefore,   an infinitesimal deformation of $((\g,[\cdot,\cdot]_\g),\rho,T)$ is determined by the triple $(\omega_1,\varrho_1,\huaT_1)$.

Now we analyze the conditions on   $(\omega_1,\varrho_1,\huaT_1)$. First by the fact that $(\R\otimes_\K\g,[\cdot,\cdot]_\g+t\omega_1)$ is an $\R$-Lie algebra, we get
  \begin{eqnarray}
   \label{eq:alg1} \dM_\CE \omega_1&=&0.
  \end{eqnarray}
   Then since $(\R\otimes_\K V;\rho+t\varrho_1)$ is a representation of    $(\R\otimes_\K\g,[\cdot,\cdot]_\g+t\omega_1)$, we obtain
   \begin{eqnarray}
    \label{eq:rep1} \rho(\omega_1(x,y))+\varrho([x,y]_\g)&=&[\rho(x),\varrho(y)]+[\varrho(x),\rho(y)].
   \end{eqnarray}
   Finally by the fact that $T+t\huaT_1$ is  an $\R$-linear relative Rota-Baxter operator on the $\R$-Lie algebra $(\R\otimes_\K\g,[\cdot,\cdot]_\g+t\omega_1)$ with respect to the representation  $(\R\otimes_\K V;\rho+t\varrho_1)$, we obtain
     \begin{eqnarray}
    \label{eq:ro1} [\huaT_1 u,Tv]_\g+[Tu,\huaT_1 v]_\g+\omega_1(Tu,Tv)&=&T\Big(\rho(\huaT_1 u)v-\rho(\huaT_1 v)u+\varrho_1(Tu)v-\varrho_1(Tv)u\Big)\\
   \nonumber  &&+\huaT_1\Big(\rho(Tu)v-\rho(Tv)u\Big).
   \end{eqnarray}

    \begin{pro}
   The triple $(\omega_1,\varrho_1,\huaT_1)$ determines an  infinitesimal  deformation  of the relative Rota-Baxter Lie algebra $((\g,[\cdot,\cdot]_\g),\rho,T)$ if and only if  $(\omega_1,\varrho_1,\huaT_1)$ is a $2$-cocycle of the relative Rota-Baxter Lie algebra $((\g,[\cdot,\cdot]_\g),\rho,T)$.
    \end{pro}
    \begin{proof}
      By \eqref{eq:alg1}, \eqref{eq:rep1} and \eqref{eq:ro1},   $(\omega_1,\varrho_1,\huaT_1)$ is a $2$-cocycle if and only if $(\omega_1,\varrho_1,\huaT_1)$ determines an  infinitesimal  deformation  of the relative Rota-Baxter Lie algebra $((\g,[\cdot,\cdot]_\g),\rho,T)$.
    \end{proof}

If two infinitesimal deformations   determined by $(\omega_1,\varrho_1,\huaT_1)$ and $(\omega_1',\varrho_1',\huaT_1')$ are equivalent, then there exists an $\R$-relative Rota-Baxter Lie algebra isomorphism $(\phi,\varphi)$ from $((\R\otimes_\K\g,[\cdot,\cdot]_\g+t\omega_1'),\rho+t\varrho_1',T+t\huaT_1')$ to $((\R\otimes_\K\g,[\cdot,\cdot]_\g+t\omega_1),\rho+t\varrho_1,T+t\huaT_1)$. By \eqref{equivalent-deformation-2}, we deduce that
   \begin{eqnarray}
   \phi=\Id_\g+tN,\quad \varphi=\Id_V+tS,\quad \mbox{where}\quad N\in\gl(\g),~  S\in\gl(V).
   \end{eqnarray}

Since $\Id_\g+tN$ is an  isomorphism from $(\R\otimes_\K\g,[\cdot,\cdot]_\g+t\omega_1')$ to $(\R\otimes_\K\g,[\cdot,\cdot]_\g+t\omega_1)$, we get
    \begin{equation}
      \label{eq:equmor1} \omega_1'-\omega_1=\dM_\CE N.
    \end{equation}
  By the equality $(\Id_V+tS)(\rho+t\varrho_1')(y)u=(\rho+t\varrho_1)\big((\Id_\g+tN)y)(\Id_V+tS)u$, we deduce that
     \begin{equation}
      \label{eq:equmor2} \varrho_1'(y)u- \varrho_1(y)u= \rho(Ny)u+ \rho(y)Su- S\rho(y)u,\quad \forall y\in\g, u\in V.
    \end{equation}
    By the equality $(\Id_\g+tN)\circ (T+t\huaT_1')=(T+t\huaT_1)\circ (\Id_V+tS)$, we obtain
    \begin{equation}
     \label{eq:equmor3}  \huaT_1'-\huaT_1=-N\circ T+T\circ S.
    \end{equation}

\begin{thm}
    There is a one-to-one correspondence between   equivalence classes of infinitesimal deformations of the relative Rota-Baxter Lie algebra $((\g,[\cdot,\cdot]_\g),\rho,T)$ and the second cohomology group $\huaH^2(\g,\rho,T)$.
    \end{thm}
   \begin{proof} By \eqref{eq:equmor1}, \eqref{eq:equmor2} and \eqref{eq:equmor3}, we deduce that
    $$
    (\omega_1',\varrho_1',\huaT_1')-(\omega_1,\varrho_1,\huaT_1)=\huaD(N,S),
    $$
    which implies that $(\omega_1,\varrho_1,\huaT_1)$ and $(\omega_1',\varrho_1',\huaT_1')$ are in the same cohomology class if and only if the corresponding infinitesimal deformations of $((\g,[\cdot,\cdot]_\g),\rho,T)$ are equivalent.
    \end{proof}

\begin{rmk}
One can study deformations of relative Rota-Baxter Lie algebras over more general bases such as $\R=\K[t]/(t^{n})$, $\R=\K[[t]]=\varprojlim_n \K[t]/(t^{n})$ or indeed over differential graded local pro-Artinian rings.
\end{rmk}

\subsection{Cohomology of Rota-Baxter Lie algebras}\label{sec:cohomologyRB}

In this subsection,  we define the cohomology of   Rota-Baxter Lie algebras with the help of the general framework of the cohomology of relative Rota-Baxter Lie algebras.

Let $(\g,[\cdot,\cdot]_\g,T)$ be a Rota-Baxter Lie algebra. We define the set of $0$-cochains $\frkC^0_{\RB}(\g,T)$ to be $0$, and define
the set of $1$-cochains $\frkC^1_{\RB}(\g,T)$ to be $\frkC^1_{\RB}(\g,T):=\Hom( \g,\g)$. For $n\geq2$, define the space of   $n$-cochains $\frkC^n_{\RB}(\g,T)$ by
\begin{eqnarray*}
\frkC^n_{\RB}(\g,T):=\frkC^n_\Lie(\g;\g)\oplus \frkC^n(T) =\Hom(\wedge^n\g,\g)\oplus\Hom(\wedge^{n-1}\g,\g).
\end{eqnarray*}

Define the embedding $\frki:\frkC^n_{\RB}(\g,T)\lon \frkC^n(\g,\ad,T)$ by
$$
\frki(f,\theta)=(f,f,\theta),\quad \forall f\in \Hom(\wedge^n\g,\g), \theta\in\Hom(\wedge^{n-1}\g,\g).
$$
Denote by $\Img^n(\frki)=\frki(\frkC^n_{\RB}(\g,T))$. Then we have

\begin{pro}\label{sub-complex-RB-cohomology}
With the above notation,  $(\oplus_{n=0}^{+\infty}\Img^n(\frki),\huaD)$ is a subcomplex of the cochain complex $(\oplus _{n=0}^{+\infty}\frkC^n(\g,\ad,T),\huaD)$ associated to the relative Rota-Baxter Lie algebra $((\g,[\cdot,\cdot]_\g),\ad,T)$.
\end{pro}

\begin{proof}
Let $(f,f,\theta)\in\Img^n(\frki)$. By the definition of $\huaD$, we have
\begin{eqnarray*}
\huaD(f,f,\theta)=((\partial(f,f))_\g,(\partial(f,f))_V,\delta \theta+h_T(f,f)).
\end{eqnarray*}
By \eqref{cohomology-algebra-rep}, we have $(\partial(f,f))_\g=\dM_\CE f$. And for any $x_1,\cdots,x_{n+1}\in\g$, we have
\begin{eqnarray*}
&&(\partial(f,f))_V(x_1,\cdots,x_{n+1})\\
&=&\sum_{1\le i<j\le n}(-1)^{i+j}f([x_i,x_j]_\g,x_1,\cdots,\hat{x}_i,\cdots,\hat{x}_j,\cdots,x_n,x_{n+1})+(-1)^{n-1}[f(x_1,\cdots,x_{n}),x_{n+1}]_\g\\
&&+\sum_{i=1}^{n}(-1)^{i+1}\Big([x_i,f(x_1,\cdots,\hat{x}_i,\cdots,x_n,x_{n+1})]_\g-f\big(x_1,\cdots,\hat{x}_i,\cdots,x_n,[x_i,x_{n+1}]_\g\big)\Big)\\
&=&(\dM_\CE f)(x_1,\cdots,x_{n+1}).
\end{eqnarray*}
Thus, we obtain  $\huaD(f,f,\theta)=(\dM_\CE f,\dM_\CE f,\delta \theta+h_T(f,f))=\frki(\dM_\CE f,\delta \theta+h_T(f,f))$, which implies that $(\oplus_n\Img^n(\frki),\huaD)$ is a subcomplex.
\end{proof}

We define the projection $\frkp:\Img^n(\frki)\lon \frkC^n_{\RB}(\g,T)$ by
$$
\frkp(f,f,\theta)=(f,\theta), \quad \forall f  \in \Hom(\wedge^n\g,\g), \theta\in\Hom(\wedge^{n-1}\g,\g).
$$
Then for $n\geq0,$ we define $\huaD_\RB:\frkC^n_{\RB}(\g,T)\lon \frkC^{n+1}_{\RB}(\g,T)$ by
$
\huaD_\RB=\frkp\circ \huaD\circ \frki.
$
More precisely,
\begin{eqnarray}\label{eq:dRB}
\huaD_\RB(f,\theta)=\Big({\dM_\CE} f, \delta \theta +\Omega f\Big),\quad \forall f\in \Hom(\wedge^n\g,\g),~\theta\in \Hom(\wedge^{n-1}\g,\g),
\end{eqnarray}
where  $\delta$ is given by \eqref{eq:odiff} and $\Omega:\Hom(\wedge^n\g,\g)\lon \Hom(\wedge^n\g,\g)$ is defined  by
\begin{eqnarray*}
(\Omega f)(x_1,\cdots,x_n)=(-1)^{n}\Big(f(Tx_1,\cdots,Tx_n)-\sum_{i=1}^{n}Tf(Tx_1,\cdots,Tx_{i-1},x_i,Tx_{i+1},\cdots,Tx_n)\Big).
\end{eqnarray*}
\begin{thm}
The map $\huaD_\RB$ is a coboundary operator, i.e. $\huaD_\RB\circ\huaD_\RB=0$.
\end{thm}
\begin{proof}
By Proposition \ref{sub-complex-RB-cohomology} and $\frki\circ\frkp=\Id$, we have
  \begin{eqnarray*}
   \huaD_\RB\circ\huaD_\RB=\frkp\circ \huaD\circ \frki\circ \frkp\circ \huaD\circ \frki=\frkp\circ \huaD\circ   \huaD\circ \frki=0,
  \end{eqnarray*}
 which finishes the proof.
\end{proof}

\begin{defi}
  Let $(\g,[\cdot,\cdot]_\g,T)$ be a Rota-Baxter Lie algebra. The cohomology of the cochain complex  $(\oplus_{n=0}^{+\infty}\frkC^n_{\RB}(\g,T),\huaD_\RB)$ is taken to be the {\em cohomology of the Rota-Baxter Lie algebra} $(\g,[\cdot,\cdot]_\g,T)$. Denote the $n$-th cohomology group by $\huaH^n_\RB(\g,T).$
\end{defi}

\begin{thm}\label{cohomology-exact-RB}
There is a short exact sequence of the  cochain complexes:
$$
0\longrightarrow(\oplus_{n=0}^{+\infty}\frkC^n(T),\delta)\stackrel{\iota}{\longrightarrow}(\oplus_{n=0}^{+\infty}\frkC^n_{\RB}(\g,T),\huaD_\RB)\stackrel{p}{\longrightarrow} (\oplus_{n=0}^{+\infty}\frkC^n_\Lie(\g;\g),\dM_\CE)\longrightarrow 0,
$$
where $\iota(\theta)=(0,\theta)$ and $p(f,\theta)=f$ for all $f\in\Hom(\wedge^n\g,\g)$ and $\theta\in \Hom(\wedge^{n-1}\g,\g)$.

Consequently,
there is a long exact sequence of the  cohomology groups:
$$
\cdots\longrightarrow\huaH^n(T)\stackrel{\huaH^n(\iota)}{\longrightarrow}\huaH^n_\RB(\g,T)\stackrel{\huaH^n(p)}{\longrightarrow} \huaH_\Lie^n(\g,\g)\stackrel{c^n}\longrightarrow \huaH^{n+1}(T)\longrightarrow\cdots,
$$
where the connecting map $c^n$ is defined by
$c^n([\alpha])=[\Omega\alpha],$  for all $[\alpha]\in \huaH_\Lie^n(\g,\g).$
\end{thm}
\begin{proof}
By  \eqref{eq:dRB},  we have the short exact sequence of cochain complexes which induces a long exact sequence of cohomology groups.
\end{proof}

\begin{rmk}
 The approach  used to define $\huaD_\RB$, can be also used to obtain the $L_\infty$-algebra structure $\{\frkl_k\}_{k=1}^{+\infty}$ on $\oplus_n\frkC^n_\RB(\g,T)$ controlling deformations of Rota-Baxter Lie algebras.  By Theorem \ref{thm:Simultaneous-deformation}, we have the $L_\infty$-algebra $(\oplus_n\frkC^n(\g,\ad,T),\{l_k\}_{k=1}^{+\infty})$ which controls deformations of the relative Rota-Baxter Lie algebra $(\g,\ad,T)$. Define $\frkl_k$ by
$$
\frkl_k(X_1,\cdots,X_k):=\frkp l_k(\frki(X_1),\cdots,\frki(X_k)),
$$
for all homogeneous elements $  X_i\in \oplus_n\frkC^n_\RB(\g,T).$ Then $(\oplus_n\frkC^n_\RB(\g ,T),\{\frkl_k\}_{k=1}^{+\infty})$ is an $L_\infty$-algebra embedded into $(\oplus_n\frkC^n(\g,\ad,T),\{l_k\}_{k=1}^{+\infty})$ as an $L_\infty$-subalgebra.
\end{rmk}

\begin{rmk}
Similarly to the study of infinitesimal deformations of relative Rota-Baxter Lie algebras, we can show that infinitesimal deformations of   Rota-Baxter Lie algebras are classified by the second cohomology group $\huaH^2_\RB(\g,R).$ We omit the details.
\end{rmk}

\subsection{Cohomology and  infinitesimal deformations of triangular Lie bialgebras}\label{sec:cohomologyTLB}
In this subsection, all vector spaces are assumed to be finite-dimensional. First we define the cohomology of   triangular  Lie bialgebras with the help of the general cohomological framework for relative Rota-Baxter Lie algebras. Then we establish the standard classification result for infinitesimal deformations of triangular Lie bialgebras using this cohomology theory.

Recall that a Lie bialgebra is a vector space $\g$ equipped
with a Lie algebra structure
$[\cdot,\cdot]_\g:\wedge^2\g\longrightarrow\g$ and a Lie coalgebra
structure $\delta:\g\longrightarrow\wedge^2\g$ such that $\delta$
is a 1-cocycle on $\g$ with coefficients in $\wedge^2\g$.
The Lie bracket $[\cdot,\cdot]_\g$ in a Lie algebra $\g$
naturally extends to the  {\em Schouten-Nijenhuis bracket} $[\cdot,\cdot]_\SN$ on $\wedge^\bullet\g=\oplus_{k\geq 0} \wedge^{k+1}\g$. More precisely, we have
$$
~[x_1\wedge\cdots \wedge x_p,y_1\wedge\cdots\wedge y_q]_\SN
=\sum_{1\le i\le p\atop 1\le j\le q} (-1)^{i+j}[x_i,y_j]_\g\wedge x_1\wedge\cdots\hat{x}_i\cdots \wedge x_p\wedge y_1\wedge\cdots\hat{y}_j\cdots\wedge y_q.
$$

An element $r\in\wedge^2\g$ is called a {\em skew-symmetric $r$-matrix} \cite{STS} if it satisfies the {\em classical Yang-Baxter equation}
$
[r,r]_\SN=0
$.
It is well known ~\cite{Ku} that $r$ satisfies the classical Yang-Baxter
equation if and only if $r^\sharp$ is a relative Rota-Baxter operator on $\g$
with respect to the coadjoint representation,
where $r^\sharp:\g^*\lon\g$ is defined by $\langle r^\sharp(\xi),\eta\rangle=\langle r,\xi\wedge \eta\rangle$ for all $\xi,\eta\in\g^*$.

Let $r$ be a skew-symmetric $r$-matrix. Define
$\delta_r:\g\longrightarrow\wedge^2\g$ by
$
\delta_r(x)=[x,r]_\SN,$  for all $x\in\g.
$
Then $(\g,[\cdot,\cdot]_\g,\delta_r)$ is a Lie bialgebra, which is
called a {\em triangular Lie bialgebra}. From now on, we will denote a triangular Lie bialgebra by $(\g,[\cdot,\cdot]_\g,r)$.

\begin{defi}\label{defi:iso}
Let $(\g_1,[\cdot,\cdot]_{\g_1},r_1)$ and $(\g_2,\{\cdot,\cdot\}_{\g_2},r_2)$ be triangular  Lie bialgebras. A linear map $\phi:\g_2\rightarrow\g_1$ is called a {\em homomorphism} from  $(\g_2,\{\cdot,\cdot\}_{\g_2},r_2)$ to $(\g_1,[\cdot,\cdot]_{\g_1},r_1)$ if $\phi$ is a Lie algebra homomorphism and
$
(\phi\otimes \phi)(r_2)=r_1$.
If $\phi$ is invertible, then $\phi$ is called
an {\em isomorphism} between triangular Lie bialgebras.
\end{defi}

The above definition is consistent with the equivalence between $r$-matrices given in \cite{CP}.

Let $\g$ be a Lie algebra and $r\in\wedge^2\g$ a skew-symmetric $r$-matrix. Define the set of 0-cochains and 1-cochains to be zero and define the set of $k$-cochains to be $\wedge^k\g$. Define $\dM_r:\wedge^k\g\lon\wedge^{k+1}\g$ by
\begin{equation}\label{eq:dr}
  \dM_r \chi =[r,\chi]_\SN,\quad\forall \chi\in \wedge^k\g.
\end{equation}
Then $\dM_r^2=0.$
Denote by $\huaH^k(r)$ the corresponding $k$-th cohomology group,
called  the {\em $k$-th cohomology group of the skew-symmetric
$r$-matrix $r$}.

For any $k\geq 1$, define $\Psi:\wedge^{k+1}\g\longrightarrow \Hom(\wedge^k\g^*,\g)$ by
\begin{equation}\label{eq:defipsi}
 \langle\Psi(\chi)(\xi_1,\cdots,\xi_k),\xi_{k+1}\rangle=\langle \chi,\xi_1\wedge\cdots\wedge\xi_k\wedge\xi_{k+1}\rangle,\quad \forall \chi\in\wedge^{k+1}\g, \xi_1,\cdots, \xi_{k+1}\in\g^*.
\end{equation}
By \cite[Theorem 7.7]{TBGS-1}, we have
\begin{equation}\label{eq:relationdd}
  \Psi(\dM_r\chi)=\delta(\Psi(\chi)),\quad \forall \chi\in\wedge^k\g.
\end{equation}
Thus $(\Img(\Psi),\delta)$ is a subcomplex of the cochain complex $(\oplus_k\frkC^k(r^\sharp),\delta)$ associated to the relative Rota-Baxter operator $r^\sharp$, where $\Img(\Psi):=\oplus_k\{\Psi(\chi)|\forall \chi\in\wedge^k\g\}$ and $\delta$ is the coboundary operator given by \eqref{defi:coboundary-O} for the relative Rota-Baxter operator $r^\sharp$.

In the following, we define the cohomology of a triangular Lie bialgebra $(\g,[\cdot,\cdot]_\g,r)$. We define the set of  $0$-cochains $\frkC^0_{\TLB}(\g,r)$ to be $0$, and define
the set of  $1$-cochains to be $\frkC^1_{\TLB}(\g,r):=\Hom(\g,\g)$. For $n\geq 2$, define the space of   $n$-cochains $\frkC^n_{\TLB}(\g,r)$ by
\begin{eqnarray*}
\frkC^n_{\TLB}(\g,r):=\Hom(\wedge^n\g,\g)\oplus \wedge^{n}\g.
\end{eqnarray*}

Define the embedding $\frki:\frkC^n_{\TLB}(\g,r)\lon \frkC^n(\g,\ad^*,r^{\sharp})=\Hom(\wedge^n\g,\g)\oplus \Hom(\wedge^{n-1}\g\otimes \g^*,\g^*)\oplus \Hom(\wedge^{n-1}\g^*,\g)$ by
$$
\frki(f,\chi)=(f,f^\star,\Psi(\chi)),\quad \forall f\in \Hom(\wedge^n\g,\g), \chi\in\wedge^n\g,
$$
where  $f^\star\in\Hom(\wedge^{n-1}\g\otimes\g^*,\g^*)$ is defined by
\begin{eqnarray}
 \label{dual-lie-r-1}\langle f^\star(x_1,\cdots,x_{n-1},\xi),x_n\rangle=-\langle \xi,f(x_1,\cdots,x_{n-1},x_n)\rangle.
\end{eqnarray}

Denote by $\Img^n(\frki)$ the image of $\frki$, i.e. $\Img^n(\frki):=\{\frki(f,\chi)|\forall (f,\chi)\in \frkC^n_{\TLB}(\g,r)\}$.

\begin{pro}
With the above notation,   $(\oplus_n\Img^n(\frki),\huaD)$ is a subcomplex of the cochain complex $(\frkC^n(\g,\ad^*,r^{\sharp}),\huaD)$ associated to the relative Rota-Baxter Lie algebra $((\g,[\cdot,\cdot]_\g),\ad^*,r^\sharp)$.
\end{pro}

\begin{proof}
Let $(f,f^\star,\Psi(\chi))\in\Img^n(\frki)$. By the definition of $\huaD$, we have
\begin{eqnarray*}
\huaD(f,f^\star,\Psi(\chi))=((\partial(f,f^\star))_\g,(\partial(f,f^\star))_{\g^*},\delta \Psi(\chi)+h_{r^\sharp}(f,f^\star)).
\end{eqnarray*}
By \eqref{cohomology-algebra-rep}, we have $(\partial(f,f^\star))_\g=\dM_\CE f$. By \eqref{cohomology-algebra-rep-V},   we can deduce that
  $(\partial(f,f^\star))_{\g^*}=(\dM_\CE f)^\star$. By \eqref{eq:odiff} and \eqref{key-cohomology-T},    we can deduce that
$
\delta \Psi(\chi)+h_{r^\sharp}(f,f^\star)\in \Img(\Psi).
$
Thus, we obtain  $$\huaD((f,f^\star,\Psi(\chi)))=(\dM_\CE f,(\dM_\CE f)^\star,\delta \Psi(\chi)+h_{r^\sharp}(f,f^\star))=\frki\Big(\dM_\CE f,\Psi^{-1}\big(\delta \Psi(\chi)+h_{r^\sharp}(f,f^\star)\big)\Big),$$ which implies that $(\oplus_n\Img^n(\frki),\huaD)$ is a subcomplex.
\end{proof}

Define the projection $\frkp:\Img^n(\frki)\lon \frkC^n_{\TLB}(\g,r)$ by
$$
\frkp(f,f^\star,\theta)=(f,\theta^\flat), \quad \forall f \in \Hom(\wedge^n\g,\g),~ \theta\in\{\Psi(\chi)|\forall \chi\in\wedge^n\g\},
$$
where $\theta^\flat\in\wedge^n\g$ is defined by$
\langle\theta^\flat,\xi_1\wedge\cdots\wedge\xi_n\rangle=\langle\theta(\xi_1,\cdots,\xi_{n-1}),\xi_n\rangle.
$
Define the {\em coboundary operator} $\huaD_\TLB:\frkC^n_{\TLB}(\g,r)\lon \frkC^{n+1}_{\TLB}(\g,r)$ for a triangular Lie bialgebra  by
$$
\huaD_\TLB=\frkp\circ \huaD\circ \frki.
$$

\begin{thm}
The map $\huaD_\TLB$ is a coboundary operator, i.e. $\huaD_\TLB\circ\huaD_\TLB=0$.
\end{thm}
\begin{proof}
 Since $\frki\circ\frkp=\Id$ when restricting on the image of $\frki$, we have
  \begin{eqnarray*}
   \huaD_\TLB\circ\huaD_\TLB=\frkp\circ \huaD\circ \frki\circ \frkp\circ \huaD\circ \frki=\frkp\circ \huaD\circ   \huaD\circ \frki=0,
  \end{eqnarray*}
 which finishes the proof.
\end{proof}

 \begin{defi}
  Let $(\g,[\cdot,\cdot]_\g,r)$ be a triangular Lie algebra. The cohomology of the cochain complex  $(\oplus_{n=0}^{+\infty}\frkC^n_{\TLB}(\g,r),\huaD_\TLB)$ is called  the {\em cohomology of the triangular Lie bialgebra} $(\g,[\cdot,\cdot]_\g,r)$. Denote the $n$-th cohomology group by $ \huaH^n_\TLB(\g,r)$.
\end{defi}

Now we give the  precise formula for the coboundary operator $\huaD_\TLB$. By the definition of $\frki$, $\frkp$, $\huaD$ and \eqref{eq:relationdd}, we have
\begin{eqnarray}\label{eq:dTLBexplicit}
\huaD_\TLB(f,\chi)=\Big({\dM}_\CE f, \Theta f+\dM_r \chi \Big),\qquad \forall f\in \Hom(\wedge^n\g,\g),~\chi\in \wedge^{n}\g,
\end{eqnarray}
where   $\dM_r$ is given by \eqref{eq:dr} and $\Theta:\Hom(\wedge^n\g,\g)\lon  \wedge^{n+1}\g $ is defined  by
$
\Theta f=\Psi^{-1}(h_{r^\sharp}(f,f^\star)).
$

The precise formula of $\Theta$ is given as follows.
\begin{lem}For any $f\in\Hom(\wedge^n\g,\g)$ and $\xi_1,\cdots,\xi_{n+1}\in\g^*$, we have
 \begin{eqnarray}
   \langle \Theta f,\xi_1\wedge \cdots\wedge\xi_{n+1}\rangle=\sum_{i=1}^{n+1}(-1)^{i+1}\langle\xi_i,f(r^{\sharp}(\xi_1),\cdots,r^{\sharp}(\xi_{i-1}),r^{\sharp}(\xi_{i+1}),
   \cdots,r^{\sharp}(\xi_{n+1}))\rangle.
 \end{eqnarray}
\end{lem}
\begin{proof} By the definition of $h_{r^\sharp}$ given by \eqref{key-cohomology-T}, we have
   \begin{eqnarray*}
   \langle \Theta f,\xi_1\wedge \cdots\wedge\xi_{n+1}\rangle&=&\Psi^{-1}(h_{r^\sharp}(f,f^\star))(\xi_1,\cdots,\xi_{n+1})\\
   &=&\langle h_{r^\sharp}(f,f^\star)(\xi_1,\cdots,\xi_{n}),\xi_{n+1}\rangle\\
   &=&   (-1)^{n}\langle f(r^{\sharp}(\xi_1),\cdots,r^{\sharp}(\xi_n)),\xi_{n+1}\rangle+\\
   &&
   \sum_{i=1}^{n}(-1)^{i+1}\langle r^{\sharp} f^\star (r^{\sharp}(\xi_1),\cdots,r^{\sharp}(\xi_{i-1}),r^{\sharp}(\xi_{i+1}),\cdots,r^{\sharp}(\xi_n),\xi_i),\xi_{n+1}\rangle\\
   &=&\sum_{i=1}^{n+1}(-1)^{i+1}\langle\xi_i,f(r^{\sharp}(\xi_1),\cdots,r^{\sharp}(\xi_{i-1}),r^{\sharp}(\xi_{i+1}),
   \cdots,r^{\sharp}(\xi_{n+1}))\rangle,
 \end{eqnarray*}
 which finishes the proof.
\end{proof}

\begin{thm}\label{cohomology-exact-r}
Let $(\g,[\cdot,\cdot]_\g,r)$ be a triangular Lie bialgebra. Then there is a short exact sequence of   cochain complexes:
$$
0\longrightarrow(\oplus_{n=2}^{+\infty}\wedge^{n}\g,\dM_r)\stackrel{\iota}{\longrightarrow}(\oplus_{n=0}^{+\infty}\frkC^n_{\TLB}(\g,r),\huaD_\TLB)\stackrel{p}{\longrightarrow} (\oplus_{n=0}^{+\infty}\frkC^n_\Lie(\g;\g),\dM_\CE)\longrightarrow 0,
$$
where $\iota(\chi)=(0,\chi)$ and $p(f,\chi)=f$ for all $\chi\in\wedge^n\g$ and $f\in\Hom(\wedge^n\g,\g)$.

Consequently, there is a long exact sequence of   cohomology groups:
\begin{equation}\label{long-exact-of-cohomology-r}
\cdots\longrightarrow\huaH^n(r)\stackrel{\huaH^n(\iota)}{\longrightarrow}\huaH^n_\TLB(\g,r)\stackrel{\huaH^n(p)}{\longrightarrow} \huaH_\Lie^n(\g,\g)\stackrel{c^n}\longrightarrow \huaH^{n+1}(r)\longrightarrow\cdots,
\end{equation}
where the connecting map $c^n$ is defined by
$c^n([\alpha])=[\Theta\alpha],$  for all $[\alpha]\in \huaH_\Lie^n(\g,\g).$
\end{thm}
\begin{proof}
 By \eqref{eq:dTLBexplicit}, we have the short exact sequence of cochain complexes which induces a long exact sequence of cohomology groups.
\end{proof}

\begin{rmk}
  In a forthcoming paper \cite{Lazarev-1}, we will use the functorial approach to give the $L_\infty$-algebra structure on $\oplus_{n=0}^{+\infty}\frkC^n_{\TLB}(\g,r)$ that control deformations of triangular Lie bialgebras, and establish the relationship with the $L_\infty$-algebra  $(\oplus_{n=0}^{+\infty}\frkC^n (\g,\ad^*,r^\sharp),\{l_k\}_{k=1}^{+\infty})$ given by Theorem \ref{thm:Simultaneous-deformation}.
\end{rmk}

We will now consider $\R$-deformations and infinitesimal deformations of triangular Lie bialgebras using the above cohomology theory, where $\R$  is a local pro-Artinian $\K$-algebra  with the augmentation $\epsilon:\R\lon \K$.

Any triangular Lie bialgebra $(\g,[\cdot,\cdot]_\g,r)$ can be viewed as a triangular $\R$-Lie bialgebra   with the help of the augmentation map $\epsilon$.

\begin{defi}
An {\em $\R$-deformation} of a triangular Lie bialgebra $(\g,[\cdot,\cdot]_\g,r)$ contains of an $\R$-Lie algebra structure $[\cdot,\cdot]_\R$ on the tensor product $\R\otimes_\K\g$ and a skew-symmetric $r$-matrix $\huaX\in(\R\otimes_\K\g)\otimes_{\R}(\R\otimes_\K\g)\cong\R\otimes_\K\g\otimes_\K\g$ such that $\epsilon\otimes_\K\Id_\g$ is an $\R$-Lie algebra homomorphism from $(\R\otimes_\K\g,[\cdot,\cdot]_\R)$ to $(\g,[\cdot,\cdot]_\g)$ and $(\epsilon\otimes_\K\Id_\g\otimes_\K\Id_\g)(\huaX)=r$.
\end{defi}

\begin{defi}
Let $(\R\otimes_\K\g,[\cdot,\cdot]_\R,\huaX)$ and $(\R\otimes_\K\g,[\cdot,\cdot]'_\R,\huaX')$ be two $\R$-deformations of a triangular Lie bialgebra $(\g,[\cdot,\cdot]_\g,r)$. We call them {\em equivalent} if there exists a triangular $\R$-Lie bialgebra isomorphism
$\phi:(\R\otimes_\K\g,[\cdot,\cdot]'_\R,\huaX')\lon(\R\otimes_\K\g,[\cdot,\cdot]_\R,\huaX)$
such that
\begin{eqnarray}\label{triangular-equivalent-deformation-2}
\epsilon\otimes_\K\Id_\g=(\epsilon\otimes_\K\Id_\g)\circ\phi.
\end{eqnarray}
\end{defi}

\begin{defi}
A $\K[t]/(t^{2})$-deformation of the triangular Lie bialgebra $(\g,[\cdot,\cdot]_\g,r)$ is called an {\em infinitesimal deformation}.
\end{defi}

Let $\R=\K[t]/(t^{2})$ and $(\R\otimes_\K\g,[\cdot,\cdot]_\R,\huaX)$ be an infinitesimal deformation of $(\g,[\cdot,\cdot]_\g,r)$. Since $(\R\otimes_\K\g,[\cdot,\cdot]_\R,\huaX)$ is a triangular $\R$-Lie bialgebra, there exist $\omega_0,\omega_1\in\Hom(\g\wedge\g,\g)$ and $\huaX_0,\huaX_1\in\wedge_{\K}^2\g$ such that
\begin{eqnarray}\label{rota-baxter-infinitesimal-deformation}
[\cdot,\cdot]_\R=\omega_0+t\omega_1,\quad \huaX=\huaX_0+t\huaX_1.
\end{eqnarray}
 Since $\epsilon\otimes_\K\Id_\g$ is an $\R$-Lie algebra homomorphism from $(\R\otimes_\K\g,[\cdot,\cdot]_\R)$ to $(\g,[\cdot,\cdot]_\g)$, we deduce that $\omega_0=[\cdot,\cdot]_\g$. By $(\epsilon\otimes_\K\Id_\g\otimes_\K\Id_\g)(\huaX)=r$, we deduce that $\huaX_0=r$. Therefore, an infinitesimal deformation of $(\g,[\cdot,\cdot]_\g,r)$  is determined by a pair $(\omega_1,\huaX_1)$.
By the fact that $(\R\otimes_\K\g,[\cdot,\cdot]_\g+t\omega_1)$ is an $\R$-Lie algebra, we get
  \begin{eqnarray}
   \label{r-matrix-1} \dM_\CE \omega_1&=&0.
  \end{eqnarray}
   Then by the fact that $r+t\huaX_1$ is a skew-symmetric $r$-matrix of the $\R$-Lie algebra $(\R\otimes_\K\g,[\cdot,\cdot]_\g+t\omega_1)$, we deduce that
      \begin{eqnarray}
   \label{r-matrix-2}2(\dM_r\huaX_1+\Theta\omega_1)=0.
   \end{eqnarray}

\begin{pro}
   The pair $(\omega_1,\huaX_1)$ determines an  infinitesimal  deformation  of the triangular Lie bialgebra $(\g,[\cdot,\cdot]_\g,r)$ if and only if  $(\omega_1,\huaX_1)$ is a $2$-cocycle of the triangular Lie bialgebra $(\g,[\cdot,\cdot]_\g,r)$, i.e. $\huaD_\TLB(\omega_1,\huaX_1)=0$.
    \end{pro}
    \begin{proof}
      By \eqref{r-matrix-1} and \eqref{r-matrix-2}, we deduce that $(\omega_1,\huaX_1)$ is a $2$-cocycle if and only if $(\omega_1,\huaX_1)$ determines an  infinitesimal  deformation  of the triangular Lie bialgebra $(\g,[\cdot,\cdot]_\g,r)$.
    \end{proof}

If two infinitesimal deformations of a triangular Lie bialgebra $(\g,[\cdot,\cdot]_\g,r)$ corresponding to $(\omega_1,\huaX_1)$ and $(\omega_1',\huaX_1')$ are equivalent, then there exists a triangular $\R$-Lie bialgebra isomorphism $\phi$ from $(\R\otimes_\K\g,[\cdot,\cdot]_\g+t\omega_1',R+t\huaX_1')$ to $(\R\otimes_\K\g,[\cdot,\cdot]_\g+t\omega_1,R+t\huaX_1)$. By \eqref{triangular-equivalent-deformation-2}, we deduce that
   \begin{eqnarray}
   \phi=\Id_\g+tN,\quad \mbox{where} \quad N\in\gl(\g).
   \end{eqnarray}
Since $\Id_\g+tN$ is an   isomorphism from $(\R\otimes_\K\g,[\cdot,\cdot]_\g+t\omega_1')$ to $(\R\otimes_\K\g,[\cdot,\cdot]_\g+t\omega_1)$, we get
    \begin{equation}
      \label{triangular-eq:equmor1} \omega_1'-\omega_1=\dM_\CE N.
    \end{equation}
    By the equality $\big((\Id_\g+tN)\otimes(\Id_\g+tN)\big)(r+t\huaX_1')=(r+t\huaX_1)$, we obtain
    \begin{equation}
     \label{triangular-eq:equmor3}  \huaX_1'-\huaX_1=-(\Id_\g\otimes N+N\otimes \Id_\g)(r)=\Theta N.
    \end{equation}

\begin{thm}
    There is a one-to-one correspondence between the space of equivalence classes of infinitesimal deformations of  $(\g,[\cdot,\cdot]_\g,r)$ and the second cohomology group $\huaH^2(\g,r).$
    \end{thm}
   \begin{proof} By \eqref{triangular-eq:equmor1} and \eqref{triangular-eq:equmor3}, we deduce that
    $$
    (\omega_1',\huaX_1')-(\omega_1,\huaX_1)=\huaD_\TLB(N),
    $$
    which implies that $(\omega_1,\huaX_1)$ and $(\omega_1',\huaX_1')$ are in the same cohomology class if and only if the corresponding  infinitesimal deformations of $(\g,[\cdot,\cdot]_\g,r)$ are equivalent.
    \end{proof}

\section{Homotopy relative Rota-Baxter Lie algebras}\label{sec:homotopy}

In this section, we introduce the notion of a homotopy relative Rota-Baxter Lie algebra, which consists of an $L_\infty$-algebra, its representation and a homotopy relative Rota-Baxter operator. We characterize homotopy relative Rota-Baxter operators as \MC elements in a certain $L_\infty$-algebra. We show that strict homotopy relative Rota-Baxter operators induce pre-Lie$_\infty$-algebras.

\subsection{Homotopy relative Rota-Baxter operators on $L_\infty$-algebras }

Denote by $\Hom^n(\bar{\Sym}(V),V)$ the space of degree $n$ linear maps from the graded vector space $\bar{\Sym}(V)=\oplus_{i=1}^{+\infty}\Sym^{i}(V)$ to the $\mathbb Z$-graded vector space $V$. Obviously, an element $f\in\Hom^n(\bar{\Sym}(V),V)$ is the sum of $f_i:\Sym^i(V)\lon V$. We will write  $f=\sum_{i=1}^{+\infty} f_i$.
 Set $C^n(V,V):=\Hom^n(\bar{\Sym}(V),V)$ and
$
C^*(V,V):=\oplus_{n\in\mathbb Z}C^n(V,V).
$
As the graded version of the Nijenhuis-Richardson bracket given in \cite{NR,NR2}, the {\em graded Nijenhuis-Richardson bracket} $[\cdot,\cdot]_{\NR}$ on the graded vector space $C^*(V,V)$ is given
by:
\begin{eqnarray}
[f,g]_{\NR}:=f\bar{\circ} g-(-1)^{mn}g\bar{\circ}f,\,\,\,\,\forall f=\sum_{i=1}^{+\infty} f_i\in C^m(V,V),~g=\sum_{j=1}^{+\infty} g_j\in C^n(V,V),
\label{eq:gfgcirc-lie}
\end{eqnarray}
where $f\bar{\circ}g\in C^{m+n}(V,V)$ is defined by
 \begin{eqnarray}\label{NR-circ}
f\bar{\circ}g&=&\Big(\sum_{i=1}^{+\infty}f_i\Big)\bar{\circ}\Big(\sum_{j=1}^{+\infty}g_j\Big):=\sum_{k=1}^{+\infty}\Big(\sum_{i+j=k+1}f_i\bar{\circ} g_j\Big),
\end{eqnarray}
while $f_i\bar{\circ} g_j\in \Hom(\Sym^{i+j-1}(V),V)$ is defined by
\begin{eqnarray}\label{graded-NR}
(f_i\bar{\circ} g_j)(v_1,\cdots,v_{i+j-1})
:=\sum_{\sigma\in\mathbb S_{(j,i-1)}}\varepsilon(\sigma)f_i(g_j(v_{\sigma(1)},\cdots,v_{\sigma(j)}),v_{\sigma(j+1)},\cdots,v_{\sigma(i+j-1)}).
\end{eqnarray}

The following result is well-known and, in fact, can be taken as a definition of an $L_\infty$-algebra.
\begin{thm}\label{graded-Nijenhuis-Richardson-bracket}
With the above notation, $(C^*(V,V),[\cdot,\cdot]_{\NR})$ is a graded Lie algebra. Its \MC elements $\sum_{k=1}^{+\infty}l_k$ are the $L_\infty$-algebra structures on $V$.\noproof
\end{thm}

\begin{defi}{\rm (\cite{LM})}
A {\em representation} of an $L_\infty$-algebra $(\g,\{l_k\}_{k=1}^{+\infty})$ on a graded vector space $V$ consists of linear maps $\rho_k:\Sym^{k-1}(\g)\otimes V\lon V$, $k\geq 1$, of degree $1$ with the property that, for any homogeneous elements $x_1,\cdots,x_{n-1}\in \g,~v\in V$, we have
\begin{eqnarray}\label{sh-Lie-rep}
&&\sum_{i=1}^{n-1}\sum_{\sigma\in \mathbb S_{(i,n-i-1)} }\varepsilon(\sigma)\rho_{n-i+1}(l_i(x_{\sigma(1)},\cdots,x_{\sigma(i)}),x_{\sigma(i+1)},\cdots,x_{\sigma(n-1)},v)\\
\nonumber&&+\sum_{i=1}^{n}\sum_{\sigma\in \mathbb S_{(n-i,i-1)} }\varepsilon(\sigma)(-1)^{x_{\sigma(1)}+\cdots+x_{\sigma(n-i)}}\rho_{n-i+1}(x_{\sigma(1)},\cdots,x_{\sigma(n-i)},\rho_i(x_{\sigma(n-i+1)},\cdots,x_{\sigma(n-1)},v))=0.
\end{eqnarray}
\end{defi}

Let $(V,\{\rho_k\}_{k=1}^{+\infty})$ be a representation of an $L_\infty$-algebra $(\g,\{l_k\}_{k=1}^{+\infty})$. There is an $L_\infty$-algebra structure on the direct sum $\g\oplus V$ given by
\begin{eqnarray*}
l_k\big((x_1,v_1),\cdots,(x_k,v_k)\big):=\big(l_k(x_1,\cdots,x_k),\sum_{i=1}^{k}(-1)^{x_i(x_{i+1}+\cdots+x_k)}\rho_k(x_1,\cdots,x_{i-1},x_{i+1},\cdots,x_k,v_i)\big).
\end{eqnarray*}
This $L_\infty$-algebra is called the
{\em semidirect product} of the $L_\infty$-algebra $(\g,\{l_k\}_{k=1}^{+\infty})$ and $(V,\{\rho_k\}_{k=1}^{+\infty})$, and denoted by $\g\ltimes_{\rho}V$.

Now we are ready to define our  main object of study in this section.

\begin{defi}\label{de:homoop}
Let $(V,\{\rho_k\}_{k=1}^{+\infty})$ be a representation of an $L_\infty$-algebra $(\g,\{l_k\}_{k=1}^{+\infty})$. A degree $0$ element $T=\sum_{k=1}^{+\infty}T_k\in \Hom(\bar{\Sym}(V),\g)$ with $T_k\in \Hom(\Sym^k(V),\g)$ is called a {\em  homotopy relative Rota-Baxter operator} on an $L_\infty$-algebra $(\g,\{l_k\}_{k=1}^{+\infty})$ with respect to the representation $(V,\{\rho_k\}_{k=1}^{+\infty})$ if the following equalities hold for all $p\geq 1$ and all homogeneous elements $v_1,\cdots,v_p\in V$,

\begin{eqnarray}
\nonumber
\label{full-homotopy-rota-baxter-o}&&\sum_{k_1+\cdots+k_m=t\atop 1\le t\le p-1}\sum_{\sigma\in \mathbb S_{(k_1,\cdots,k_m,1,p-1-t)}}\frac{\varepsilon(\sigma)}{m!}\cdot\\
\nonumber &&T_{p-t}\Big(\rho_{m+1}\Big(T_{k_1}\big(v_{\sigma(1)},\cdots,v_{\sigma(k_1)}\big),\cdots,T_{k_m}\big(v_{\sigma(k_1+\cdots+k_{m-1}+1)},\cdots,v_{\sigma(t)}\big),v_{\sigma(t+1)}\Big),v_{\sigma(t+2)},\cdots,v_{\sigma(p)}\Big)\\
&=&\sum_{k_1+\cdots+k_n=p}\sum_{\sigma\in \mathbb S_{(k_1,\cdots,k_n)}}\frac{\varepsilon(\sigma)}{n!}l_n\Big(T_{k_1}\big(v_{\sigma(1)},\cdots,v_{\sigma(k_1)}\big),\cdots,T_{k_n}\big(v_{\sigma(k_1+\cdots+k_{n-1}+1)},\cdots,v_{\sigma(p)}\big)\Big).
\nonumber
\end{eqnarray}

\end{defi}

A homotopy relative Rota-Baxter operator on an  $L_\infty$-algebra is a generalization of an $\huaO$-operator on a Lie $2$-algebra introduced in \cite{Sheng}.

\begin{defi}
Let $(\g,\{l_k\}_{k=1}^{+\infty})$ be an $L_\infty$-algebra. A degree $0$ element $T=\sum_{k=1}^{+\infty}T_k\in \Hom(\bar{\Sym}(\g),\g)$ with $T_k\in \Hom(\Sym^k(\g),\g)$ is called a {\em  homotopy Rota-Baxter operator } on an $L_\infty$-algebra $(\g,\{l_k\}_{k=1}^{+\infty})$ if the following equalities hold for all $p\geq 1$ and all homogeneous elements $x_1,\cdots,x_p\in \g$,
\begin{eqnarray}
\nonumber
&&\sum_{k_1+\cdots+k_m=t\atop 1\le t\le p-1}\sum_{\sigma\in \mathbb S_{(k_1,\cdots,k_m,1,p-1-t)}}\frac{\varepsilon(\sigma)}{m!}\cdot\\
\nonumber&&T_{p-t}\Big(l_{m+1}\Big(T_{k_1}\big(x_{\sigma(1)},\cdots,x_{\sigma(k_1)}\big),\cdots,T_{k_m}\big(x_{\sigma(k_1+\cdots+k_{m-1}+1)},\cdots,x_{\sigma(t)}\big),x_{\sigma(t+1)}\Big),x_{\sigma(t+2)},\cdots,x_{\sigma(p)}\Big)\\
&=&\sum_{k_1+\cdots+k_n=p}\sum_{\sigma\in \mathbb S_{(k_1,\cdots,k_n)}}\frac{\varepsilon(\sigma)}{n!}l_n\Big(T_{k_1}\big(x_{\sigma(1)},\cdots,x_{\sigma(k_1)}\big),\cdots,T_{k_n}\big(x_{\sigma(k_1+\cdots+k_{n-1}+1)},\cdots,x_{\sigma(p)}\big)\Big).
\nonumber
\end{eqnarray}

\end{defi}

\begin{rmk}
A homotopy Rota-Baxter operator $T=\sum_{k=1}^{+\infty}T_k\in \Hom(\bar{\Sym}(\g),\g)$ on an $L_\infty$-algebra $(\g,\{l_k\}_{k=1}^{+\infty})$ is a homotopy relative Rota-Baxter operator with respect to the adjoint representation.
If moreover the $L_\infty$-algebra reduces to a Lie algebra $(\g,[\cdot,\cdot]_\g)$, then the resulting linear operator $T:\g\longrightarrow \g$ is a {\em Rota-Baxter operator}.
\end{rmk}

\begin{defi}\label{defi:homotopy-O}
\begin{enumerate}
\item[\rm(i)] An $L_\infty$-algebra $(\g,\{l_k\}_{k=1}^{+\infty})$ with a homotopy Rota-Baxter operator $T=\sum_{k=1}^{+\infty}T_k\in \Hom(\bar{\Sym}(\g),\g)$ is
called a {\em homotopy Rota-Baxter Lie algebra}. We denote it by $\big(\g,\{l_k\}_{k=1}^{+\infty},\{T_k\}_{k=1}^{+\infty}\big)$.
\item[\rm(ii)] A {\em homotopy relative Rota-Baxter Lie algebra} is a triple $\big((\g,\{l_k\}_{k=1}^{+\infty}),\{\rho_k\}_{k=1}^{+\infty},\{T_k\}_{k=1}^{+\infty}\big)$, where $(\g,\{l_k\}_{k=1}^{+\infty})$ is an $L_\infty$-algebra, $(V,\{\rho_k\}_{k=1}^{+\infty})$ is a representation of $\g$ on a graded vector space $V$ and $T=\sum_{k=1}^{+\infty}T_k\in \Hom(\bar{\Sym}(V),\g)$ is a  homotopy relative Rota-Baxter operator.
\end{enumerate}
\end{defi}

A representation of an $L_\infty$-algebra will give rise to a V-data as well as an $L_\infty$-algebra that characterize homotopy relative Rota-Baxter operators  as \MC elements.

\begin{pro}
Let $(\g,\{l_k\}_{k=1}^{+\infty})$ be an $L_\infty$-algebra and $(V,\{\rho_k\}_{k=1}^{+\infty})$ a representation of $(\g,\{l_k\}_{k=1}^{+\infty})$. Then the following quadruple forms a V-data:
\begin{itemize}
\item[$\bullet$] the graded Lie algebra $(L,[\cdot,\cdot])$ is given by $(C^*(\g\oplus V,\g\oplus V),[\cdot,\cdot]_{\NR})$;
\item[$\bullet$] the abelian graded Lie subalgebra $\h$ is given by $\h:=\oplus_{n\in\mathbb Z}\Hom^n(\bar{\Sym}(V),\g);$
\item[$\bullet$] $P:L\lon L$ is the projection onto the subspace $\h$;
\item[$\bullet$] $\Delta=\sum_{k=1}^{+\infty}(l_k+\rho_k)$.
\end{itemize}

Consequently, $(\h,\{\frkl_k\}_{k=1}^{+\infty})$ is an $L_\infty$-algebra, where $\frkl_k$ is given by
\eqref{V-shla}.
\end{pro}

\begin{proof}
By Theorem \ref{graded-Nijenhuis-Richardson-bracket}, we obtain that $(C^*(\g\oplus V,\g\oplus V),[\cdot,\cdot]_{\NR})$ is a graded Lie algebra. Moreover, by \eqref{graded-NR} we deduce that $\Img P=\h$ is an abelian graded Lie subalgebra and $\ker P$ is a graded Lie subalgebra. Since $\Delta=\sum_{k=1}^{+\infty}(l_k+\rho_k)$ is the semidirect product  $L_\infty$-algebra structure on $\g\oplus V$, we have $[\Delta,\Delta]_{\NR}=0$ and $P(\Delta)=0$. Thus $(L,\h,P,\Delta)$ is a V-data. Hence by Theorem \ref{thm:db}, we obtain the higher derived brackets $\{{\frkl_k}\}_{k=1}^{+\infty}$ on the abelian graded Lie subalgebra $\h$.
\end{proof}

Moreover, for all $n\ge 1$, we set
\begin{eqnarray}\label{filtration-homotopy-rota-baxter}
\quad\huaF_n(\h)=\Pi_{i=n}^{+\infty}\Hom(\Sym^i(V),\g).
\end{eqnarray}

\begin{lem}\label{filtered-homotopy-lie-of-rota-baxter}
With above notation, $(\h,\{\frkl_k\}_{k=1}^{+\infty})$ is a weakly filtered $L_\infty $-algebra.
\end{lem}

\begin{proof}
By \eqref{filtration-homotopy-rota-baxter}, we have $\h=\huaF_1( \h)\supset\cdots\supset\huaF_n( \h)\supset\cdots$. Moreover, by Lemma \ref{important-lemma-2}, we have
\begin{eqnarray}
\frkl_k(\huaF_{n_1}(\h),\huaF_{n_2}(\h),\cdots,\huaF_{n_k}(\h))\subset\huaF_{n_1+n_2+\cdots+n_k}(\h) \subset\huaF_{k}(\h).
\end{eqnarray}
Thus, we deduce that $\big(\h,\huaF_{\bullet}(\h)\big)$ is a weakly filtered $L_\infty $-algebra with $n=1$.
\end{proof}
\begin{rmk}
	In fact, the above argument shows that $\big(\h,\huaF_{\bullet}(\h)\big)$ is a \emph{filtered} $L_\infty $-algebra in the sense of \cite{Dolgushev-Rogers}.
\end{rmk}
\begin{thm}\label{hmotopy-o-operator-homotopy-lie}
With the above notation, a degree $0$ element  $T=\sum_{k=1}^{+\infty}T_k\in \Hom(\bar{\Sym}(V),\g)$ is a homotopy relative Rota-Baxter operator on $(\g,\{l_k\}_{k=1}^{+\infty})$ with respect to the representation $(V,\{\rho_k\}_{k=1}^{+\infty})$ if and only if $T=\sum_{k=1}^{+\infty}T_k$ is an \MC element of the  $L_\infty$-algebra $(\h,\{{\frkl_k}\}_{k=1}^{+\infty})$.
\end{thm}

\begin{proof}
By Remark \ref{NR-coder}, we will view the elements of $C^*(\g\oplus V,\g\oplus V)$ as coderivations of $\bar{\Sym}^c(\g\oplus V)$. Moreover,  we view $\oplus_{n\in\mathbb Z}\Hom^n(\bar{\Sym}(V),\g)$ as an abelian graded Lie subalgebra of the graded Lie algebra $\Coder(\bar{\Sym}^c(\g\oplus V))$ and we denote by $\bar{P}$ the projection onto this  Lie subalgebra. The coderivations of $\bar{\Sym}^c(\g\oplus V)$ corresponding to $\sum_{k=1}^{+\infty}l_k,~ \sum_{k=1}^{+\infty}\rho_k$ and $\sum_{k=1}^{+\infty}T_k$ will be denoted by $\bar{l},~\bar{\rho}$ and $\bar{T}$ respectively. Then $T=\sum_{k=1}^{+\infty}T_k$ is an \MC element of the $L_\infty$-algebra $(\h,\{{\frkl_k}\}_{k=1}^{+\infty})$ if and only if
 \begin{equation}\label{mc-homotopy-o}
 \bar{P}\sum_{n=1}^{+\infty}\frac{1}{n!}\underbrace{[\cdots[[}_{ n}\bar{l}+\bar{\rho},\bar{T}],\bar{T}],\cdots,\bar{T}]=0.
 \end{equation}
 In fact, we have
 \begin{eqnarray*}
\underbrace{[\cdots[[}_{ n}\bar{l}+\bar{\rho},\bar{T}],\bar{T}],\cdots,\bar{T}]=\sum_{i=0}^{n}(-1)^{i}{n\choose i}\big(\underbrace{\bar{T}\circ\cdots\circ\bar{T}}_i\circ (\bar{l}+\bar{\rho})\circ \underbrace{\bar{T}\circ \cdots\circ \bar{T}}_{n-i}\big).
 \end{eqnarray*}
We denote by $\pr_\g$ the natural projections from $\bar{\Sym}(\g\oplus V)$ to $\g$. Thus, for all $v_1,\cdots,v_{p}\in V$, we have
 \begin{eqnarray*}
 &&\big(\pr_\g\circ\underbrace{[\cdots[[}_{ n}\bar{l}+\bar{\rho},\bar{T}],\bar{T}],\cdots,\bar{T}]\big)(v_1,\cdots,v_{p})\\
 &=&\big(\pr_\g\circ\bar{l}\circ \underbrace{\bar{T}\cdots\circ \bar{T}}_{n}\big)(v_1,\cdots,v_{p})-n\big(\pr_\g\circ\bar{T}\circ\bar{\rho}\circ\underbrace{\bar{T}\cdots\circ \bar{T}}_{n-1}\big)(v_1,\cdots,v_{p}).
 \end{eqnarray*}
 By \eqref{graded-NR}, we obtain that
 \begin{eqnarray*}
 &&\big(\pr_\g\circ\bar{l}\circ \underbrace{\bar{T}\cdots\circ \bar{T}}_{n}\big)(v_1,\cdots,v_{p})\\
 &=&\sum_{k_1+\cdots+k_n=p}\sum_{\sigma\in \mathbb S_{(k_1,\cdots,k_n)}}\varepsilon(\sigma)l_n\Big(T_{k_1}\big(v_{\sigma(1)},\cdots,v_{\sigma(k_1)}\big),\cdots,T_{k_n}\big(v_{\sigma(k_1+\cdots+k_{n-1}+1)},\cdots,v_{\sigma(p)}\big)\Big)
 \end{eqnarray*}
and

 \begin{eqnarray*}
 &&n\big(\pr_\g\circ\bar{T}\circ\bar{\rho}\circ\underbrace{\bar{T}\cdots\circ \bar{T}}_{n-1}\big)(v_1,\cdots,v_{p})=n\sum_{k_1+\cdots+k_{n-1}=t\atop 1\le t\le p-1}\sum_{\tau\in \mathbb S_{(k_1,\cdots,k_{n-1},p-t)}}\varepsilon(\tau)\cdot\\&&\Big(\pr_\g\circ\bar{T}\circ\bar{\rho}\Big)\Big(T_{k_1}\big(v_{\tau(1)},\cdots,v_{\tau(k_1)}\big),\cdots,T_{k_{n-1}}\big(v_{\tau(k_1+\cdots+k_{n-2}+1)},\cdots,v_{\tau(t)}\big),v_{\tau(t)+1)},\cdots,v_{\tau(p)}\Big)\\
 &=&n\sum_{k_1+\cdots+k_{n-1}=t\atop 1\le t\le p-1}\sum_{\sigma\in \mathbb S_{(k_1,\cdots,k_{n-1},1,p-1-t)}}\varepsilon(\sigma)\cdot\\
 &&T_{p-t}\Big(\rho_n\Big(T_{k_1}\big(v_{\sigma(1)},\cdots,v_{\sigma(k_1)}\big),\cdots,T_{k_{n-1}}\big(v_{\sigma(k_1+\cdots+k_{n-2}+1)},\cdots,v_{\sigma(t)}\big),v_{\sigma(t+1)}\Big),v_{\sigma(t+2)},\cdots,v_{\sigma(p)}\Big).
 \end{eqnarray*}

Thus,  \eqref{mc-homotopy-o} holds if and only if $T=\sum_{k=1}^{+\infty}T_k\in \Hom(\bar{\Sym}(V),\g)$ is a homotopy relative Rota-Baxter operator on $(\g,\{l_k\}_{k=1}^{+\infty})$ with respect to the representation $(V,\{\rho_k\}_{k=1}^{+\infty})$.
\end{proof}

At the end of this section, we show that a homotopy relative Rota-Baxter operator corresponding to a representation $V$ naturally gives rise to a $L_\infty$ structure on $V$.

\begin{pro}\label{twist-homotopy-lie}
Let $T=\sum_{k=1}^{+\infty}T_k\in \Hom(\bar{\Sym}(V),\g)$ be a homotopy relative Rota-Baxter operator on $(\g,\{l_k\}_{k=1}^{+\infty})$ with respect to the representation $(V,\{\rho_k\}_{k=1}^{+\infty})$.
\begin{itemize}
  \item[\rm(i)] $e^{[\cdot,T]_\NR}\Big(\sum_{k=1}^{+\infty}(l_k+\rho_k)\Big)$ is an  \MC element of the graded Lie algebra $(C^*(\g\oplus V,\g\oplus V),[\cdot,\cdot]_{\NR})$;
  \item[\rm(ii)] there is an $L_\infty$-algebra structure on $V$  given by
\begin{eqnarray}\label{double-homotopy-lie}
\nonumber&&\frkl_{t+1}(v_1,\cdots,v_{t+1})=\sum_{k_1+\cdots+k_m=t}\sum_{\sigma\in \mathbb S_{(k_1,\cdots,k_m,1)}}\\
&&\frac{\varepsilon(\sigma)}{m!}\rho_{m+1}\Big(T_{k_1}\big(v_{\sigma(1)},\cdots,v_{\sigma(k_1)}\big),\cdots,T_{k_m}\big(v_{\sigma(k_1+\cdots+k_{m-1}+1)},\cdots,v_{\sigma(t)}\big),v_{\sigma(t+1)}\Big);
\end{eqnarray}
  \item[\rm(iii)]  $T$ is an $L_\infty$-algebra homomorphism from the  $L_\infty$-algebra $(V,\{\frkl_k\}_{k=1}^{+\infty})$ to   $(\g,\{l_k\}_{k=1}^{+\infty})$.
\end{itemize}
\end{pro}

\begin{proof}
(i) For any  $\frkX\in \Hom(\Sym^i(\g)\otimes \Sym^j(V),\g)$ and $\frkY\in \Hom(\Sym^i(\g)\otimes \Sym^j(V),V)$, we have \begin{eqnarray*}
 \underbrace{[\cdots[[}_{i+1}\frkX,T]_\NR,T]_\NR,\cdots,T]_\NR&=&0,\\
 \underbrace{[\cdots[[}_{i+2}\frkY,T]_\NR,T]_\NR,\cdots,T]_\NR&=&0.
\end{eqnarray*}  Thus,   $[\cdot,T]_\NR$ is a locally nilpotent derivation of $(C^*(\g\oplus V,\g\oplus V),[\cdot,\cdot]_{\NR})$. Since $e^{[\cdot,T]_\NR}$ is an automorphism of $(C^*(\g\oplus V,\g\oplus V),[\cdot,\cdot]_{\NR})$, we have
\begin{eqnarray*}
[e^{[\cdot,T]_\NR}\Big(\sum_{k=1}^{+\infty}(l_k+\rho_k)\Big),e^{[\cdot,T]_\NR}\Big(\sum_{k=1}^{+\infty}(l_k+\rho_k)\Big)]_\NR=e^{[\cdot,T]_\NR}[\sum_{k=1}^{+\infty}(l_k+\rho_k),\sum_{k=1}^{+\infty}(l_k+\rho_k)]_\NR=0,
\end{eqnarray*}
which implies that $e^{[\cdot,T]_\NR}\Big(\sum_{k=1}^{+\infty}(l_k+\rho_k)\Big)$ is an  \MC element.

(ii)  By \eqref{mc-homotopy-o},   $e^{[\cdot,T]_\NR}\Big(\sum_{k=1}^{+\infty}(l_k+\rho_k)\Big)|_V$ is a sub $L_\infty$-algebra structure on $V$. It is straightforward to deduce that the $L_\infty$-algebra structure on $V$ is just given by \eqref{double-homotopy-lie}.

(iii) By the definition of  a homotopy relative Rota-Baxter operator and  \eqref{double-homotopy-lie}, we deduce that $T$ is an $L_\infty$-algebra homomorphism.
\end{proof}

\subsection{Strict homotopy relative Rota-Baxter operators on $L_\infty$-algebras and pre-Lie$_\infty$-algebras}

\begin{defi}
Let $(V,\{\rho_k\}_{k=1}^{+\infty})$ be a representation of an $L_\infty$-algebra $(\g,\{l_k\}_{k=1}^{+\infty})$. A degree $0$ element $T\in \Hom(V,\g)$ is called a {\em strict homotopy relative Rota-Baxter operator} on an $L_\infty$-algebra $(\g,\{l_k\}_{k=1}^{+\infty})$ with respect to the representation $(V,\{\rho_k\}_{k=1}^{+\infty})$ if the following equalities hold for all $p\geq 1$ and all homogeneous elements $v_1,\cdots,v_p\in V$,
\begin{eqnarray}
\label{rota-baxter-o-homotopy-lie}
l_p\big(Tv_{1},\cdots,Tv_{p}\big)=\sum_{i=1}^{p}(-1)^{(v_{i+1}+\cdots+v_p)v_i}T\rho_{p}(Tv_1,\cdots,Tv_{i-1},Tv_{i+1},\cdots,Tv_p,v_i).
\end{eqnarray}
\end{defi}

\begin{rmk}
A strict homotopy relative Rota-Baxter operator   is just a homotopy relative Rota-Baxter operator $T=\sum_{i=1}^{+\infty}T_i\in \Hom(\bar{\Sym}(V),\g)$, in which $T_i=0$ for all $i\ge 2$.
\end{rmk}

Let $V$ be a graded vector space. Denote by $\Hom^n(\Sym(V)\otimes V,V)$ the space of degree $n$ linear maps from the graded vector space $\Sym(V)\otimes V$ to the graded vector space $V$. Obviously, an element $f\in\Hom^n(\Sym(V)\otimes V,V)$ is the sum of $f_i:\Sym^{i-1}(V)\otimes V\lon V$. We will write  $f=\sum_{i=1}^{+\infty} f_i$.
 Set $\CV^n(V,V):=\Hom^n(\Sym(V)\otimes V,V)$ and
$
\CV^*(V,V):=\oplus_{n\in\mathbb Z}\CV^n(V,V).
$
As the graded version of the Matsushima-Nijenhuis bracket given in \cite{CL}, the {\em graded Matsushima-Nijenhuis bracket} $[\cdot,\cdot]_{\MN}$ on the graded vector space $\CV^*(V,V)$ is given
by:
\begin{eqnarray}
[f,g]_{\MN}:=f\diamond g-(-1)^{mn}g\diamond f,\,\,\,\,\forall f=\sum_{i=1}^{+\infty} f_i\in \CV^m(V,V),~g=\sum_{j=1}^{+\infty} g_j\in \CV^n(V,V),
\label{eq:gfgcirc}
\end{eqnarray}
where $f\diamond g\in \CV^{m+n}(V,V)$ is defined by
 \begin{eqnarray}\label{graded-NR-circ}
f\diamond g&=&\Big(\sum_{i=1}^{+\infty}f_i\Big)\diamond\Big(\sum_{j=1}^{+\infty}g_j\Big):=\sum_{k=1}^{+\infty}\Big(\sum_{i+j=k+1}f_i\diamond g_j\Big),
\end{eqnarray}
while $f_i\diamond g_j\in \Hom(\Sym^{i+j-2}(V)\otimes V,V)$ is defined by
\begin{eqnarray}
\nonumber&&(f_i\diamond g_j)(v_1,\cdots,v_{i+j-1})\\
\label{graded-MN}&=&\sum_{\sigma\in\mathbb S_{(j-1,1,i-2)}}\varepsilon(\sigma)f_i(g_j(v_{\sigma(1)},\cdots,v_{\sigma(j-1)},v_{\sigma(j)}),v_{\sigma(j+1)},\cdots,v_{\sigma(i+j-2)},v_{i+j-1})\\
\nonumber&&+\sum_{\sigma\in\mathbb S_{(i-1,j-1)}}(-1)^{\alpha}\varepsilon(\sigma)f_i(v_{\sigma(1)},\cdots,v_{\sigma(i-1)},g_j( v_{\sigma(i)},\cdots,v_{\sigma(i+j-2)},v_{i+j-1})),
\end{eqnarray}
where $\alpha=n(v_{\sigma(1)}+v_{\sigma(2)}+\cdots+v_{\sigma(i-1)})$. Then the graded vector space $\CV^*(V,V)$ equipped with the  graded Matsushima-Nijenhuis bracket $[\cdot,\cdot]_{\MN}$ is a graded Lie algebra.

The notion of a pre-Lie$_\infty$-algebra was introduced in \cite{CL}. See \cite{Mer} for more applications of pre-Lie$_\infty$-algebras in geometry.

\begin{thm}{\rm (\cite{CL})}\label{MC-of-homotopy-pre-lie}
  Let $\sum_{k=1}^{+\infty}\oprn_k$ be a degree $1$ linear map from the graded vector space $S(V)\otimes V$ to the graded vector space $V$. Then $(V,\{\oprn_k\}_{k=1}^{+\infty})$ is a pre-Lie$_\infty$-algebra if and only if $\sum_{k=1}^{+\infty}\oprn_k$ is an \MC   element of the graded Lie algebra $(\CV^*(V,V),[\cdot,\cdot]_{\MN})$.\qed
\end{thm}

Now we show that there is a close relationship between  the graded Lie algebra $(\CV^*(V,V),[\cdot,\cdot]_{\MN})$ and $(C^*(V,V),[\cdot,\cdot]_{\NR})$.
 Define a graded linear map $\Phi:\CV^*(V,V)\lon C^*(V,V)$ of degree $0$ by
$$
\Phi(f)=\sum_{k=1}^{+\infty}\Phi(f)_k=\Phi(f_k),\quad \forall f=\sum_{k=1}^{+\infty}f_k\in\Hom^m(\Sym(V)\otimes V,V),
$$
where $ \Phi(f_k)$  is given by
\begin{eqnarray*}\label{homotopy-pre-to-lie}
\Phi(f_k)(v_1,\cdots,v_k)=\sum_{\sigma\in\mathbb S_{(k-1,1)}}\varepsilon(\sigma)f_k(v_{\sigma(1)},\cdots,v_{\sigma(k)})=\sum_{i=1}^{k}(-1)^{v_i(v_{i+1}+\cdots+v_k)}f_k(v_{1},\cdots,\hat{v}_{i},\cdots,v_{k},v_i).
\end{eqnarray*}

\begin{thm}\label{homotopy-pre-lie-to-homotopy-lie}
 $\Phi$ is a homomorphism   from the graded Lie algebra
$(\CV^*(V,V),[\cdot,\cdot]_{\MN})$ to the graded Lie algebra
$(C^*(V,V),[\cdot,\cdot]_{\NR})$.
\end{thm}

\begin{proof}It follows from a direct but tedious computation. We omit details.
\end{proof}

In the classical case, the symmetrization of a pre-Lie algebra gives rise to a Lie algebra. The following result generalizes this construction to pre-Lie$_\infty$-algebras and $L_\infty$-algebras.

\begin{cor}\label{homotopy-pre-lie-to-lie-cor}
Let $(\g,\{\oprn_k\}_{k=1}^{+\infty})$ be a pre-Lie$_\infty$-algebra and we define $l_k$ by
\begin{eqnarray}\label{homotopy-pre-lie-to-lie}
l_k(x_1,\cdots,x_k)=\Phi(\oprn_k)(x_1,\cdots,x_k)=\sum_{i=1}^{k}(-1)^{x_i(x_{i+1}+\cdots+x_k)}\oprn_k(x_{1},\cdots,\hat{x}_{i},\cdots,x_{k},x_i).
\end{eqnarray}
Then $(\g,\{l_k\}_{k=1}^{+\infty})$ is an $L_\infty$-algebra. We denote this $L_\infty$-algebra by $\g^C$ and call it  the {\em sub-adjacent $L_\infty$-algebra} of $(\g,\{\oprn_k\}_{k=1}^{+\infty})$. Moreover, $(\g,\{\oprn_k\}_{k=1}^{+\infty})$ is called the {\em compatible pre-Lie$_\infty$-algebra} structure on the $L_\infty$-algebra $\g^C$.
\end{cor}
\begin{proof}
 It follows from Theorem \ref{graded-Nijenhuis-Richardson-bracket}, Theorem \ref{MC-of-homotopy-pre-lie} and Theorem \ref{homotopy-pre-lie-to-homotopy-lie}.
\end{proof}

Let $(\g,\{\oprn_k\}_{k=1}^{+\infty})$ be a pre-Lie$_\infty$-algebra. For all $k\ge 1$, we define  $L_k:\Sym^{k-1}(\g)\otimes \g\lon \g$ by
\begin{eqnarray}\label{left-multiplication}
L_k(x_1,\cdots,x_{k-1},x_k)=\oprn_k(x_1,\cdots,x_{k-1},x_k).
\end{eqnarray}

\begin{pro}\label{homotopy-pre-lie-id-o}
With the above notation, $(\g,\{L_k\}_{k=1}^{+\infty})$ is a representation of the sub-adjacent $L_\infty$-algebra $\g^C$. Moreover, the identity map $\Id:\g\lon\g$ is a strict homotopy relative Rota-Baxter operator on the $L_\infty$-algebra $\g^C$ with respect to the representation $(\g,\{L_k\}_{k=1}^{+\infty})$.
\end{pro}

\begin{proof}
For all $x_1,\cdots,x_n\in\g$, by the definition of pre-Lie$_\infty$-algebras, we have
\begin{eqnarray*}
&&\sum_{i=1}^{n-1}\sum_{\sigma\in \mathbb S_{(i,n-i-1)} }\varepsilon(\sigma)L_{n-i+1}(l_i(x_{\sigma(1)},\cdots,x_{\sigma(i)}),x_{\sigma(i+1)},\cdots,x_{\sigma(n-1)},x_n)\\
&&+\sum_{i=1}^{n}\sum_{\sigma\in \mathbb S_{(n-i,i-1)} }\varepsilon(\sigma)(-1)^{x_{\sigma(1)}+\cdots+x_{\sigma(n-i)}}L_{n-i+1}(x_{\sigma(1)},\cdots,x_{\sigma(n-i)},L_i(x_{\sigma(n-i+1)},\cdots,x_{\sigma(n-1)},x_n))\\
&\stackrel{\eqref{homotopy-pre-lie-to-lie},\eqref{left-multiplication}}{=}&
\sum_{i=1}^{n-1}\sum_{\tau\in \mathbb S_{(i-1,1,n-i-1)} }\varepsilon(\tau)\oprn_{n-i+1}(\oprn_i(x_{\tau(1)},\cdots,x_{\tau(i-1)},x_{\tau(i)}),x_{\tau(i+1)},\cdots,x_{\tau(n-1)},x_n)\\
&&+\sum_{i=1}^{n}\sum_{\tau\in \mathbb S_{(n-i,i-1)} }\varepsilon(\tau)(-1)^{x_{\tau(1)}+\cdots+x_{\tau(n-i)}}\oprn_{n-i+1}(x_{\tau(1)},\cdots,x_{\tau(n-i)},\oprn_i(x_{\tau(n-i+1)},\cdots,x_{\tau(n-1)},x_n))\\
& {=}&0.
\end{eqnarray*}
Thus, we deduce that $(\g,\{L_k\}_{k=1}^{+\infty})$ is a representation of the sub-adjacent $L_\infty$-algebra $\g^C$.
By \eqref{homotopy-pre-lie-to-lie}, we deduce that $\Id$ is a strict homotopy relative Rota-Baxter operator on   $\g^C$ with respect to  $(\g,\{L_k\}_{k=1}^{+\infty})$.

\end{proof}

Now we are ready to show that   strict homotopy relative Rota-Baxter operators on an $L_\infty$-algebra $(\g,\{l_k\}_{k=1}^{+\infty})$ induce pre-Lie$_\infty$-algebras. This generalizes the   result given in  \cite{Bai}.

\begin{thm}\label{relative-o-to-homotopy-pre-lie-thm}
Let $T\in \Hom(V,\g)$ be a strict homotopy relative Rota-Baxter operator on an $L_\infty$-algebra $(\g,\{l_k\}_{k=1}^{+\infty})$ with respect to the representation $(V,\{\rho_k\}_{k=1}^{+\infty})$. Then $(V,\{\oprn_k\}_{k=1}^{+\infty})$ is a pre-Lie$_\infty$-algebra, where  $\oprn_k:\otimes^k V\lon V$ $(k\ge 1)$  are   linear maps  of degree $1$ defined by
\begin{eqnarray}\label{relative-o-to-homotopy-pre-lie}
\oprn_k(v_1,\cdots,v_k):=\rho_k(Tv_1,\cdots,Tv_{k-1},v_k),\quad \forall v_1\cdots,v_k\in V.
\end{eqnarray}
\end{thm}

\begin{proof}
By the fact that $\rho_k$ is a linear map of degree $1$ from graded vector space $\Sym^{k-1}(\g)\otimes V$ to $V$, we deduce the graded symmetry condition of $\oprn_k$. Moreover, for all $v_1\cdots,v_n\in V$, we have
\begin{eqnarray*}
&&\sum_{i+j=n+1\atop i\ge1,j\geq2}\sum_{\sigma\in\mathbb S_{(i-1,1,j-2)}}\varepsilon(\sigma)\oprn_j(\oprn_i(v_{\sigma(1)},\cdots,v_{\sigma(i-1)},v_{\sigma(i)}), v_{\sigma(i+1)},\cdots,v_{\sigma(n-1)},v_{n})\\
&&+\sum_{i+j=n+1\atop i\geq1,j\geq1}\sum_{\sigma\in\mathbb S_{(j-1,i-1)}}(-1)^{v_{\sigma(1)}+v_{\sigma(2)}+\cdots+v_{\sigma(j-1)}}\varepsilon(\sigma)\oprn_j(v_{\sigma(1)},\cdots,v_{\sigma(j-1)},\oprn_i( v_{\sigma(j)},\cdots,v_{\sigma(n-1)},v_{n}))\\
&\stackrel{\eqref{relative-o-to-homotopy-pre-lie}}{=}&\sum_{i+j=n+1\atop i\ge1,j\geq2}\sum_{\sigma\in\mathbb S_{(i-1,1,j-2)}}\varepsilon(\sigma)\rho_j\big(T\rho_i(Tv_{\sigma(1)},\cdots,Tv_{\sigma(i-1)},v_{\sigma(i)}), Tv_{\sigma(i+1)},\cdots,Tv_{\sigma(n-1)},v_{n}\big)\\
&&+\sum_{i+j=n+1\atop i\geq1,j\geq1}\sum_{\sigma\in\mathbb S_{(j-1,i-1)}}(-1)^{v_{\sigma(1)}+v_{\sigma(2)}+\cdots+v_{\sigma(j-1)}}\varepsilon(\sigma)\rho_j\big(Tv_{\sigma(1)},\cdots,Tv_{\sigma(j-1)},\rho_i( Tv_{\sigma(j)},\cdots,Tv_{\sigma(n-1)},v_{n})\big)\\
&=&\sum_{i+j=n+1\atop i\ge1,j\geq2}\sum_{\tau\in\mathbb S_{(i,j-2)}}\sum_{s=1}^{i}(-1)^{v_{\tau(s)}(v_{\tau(s+1)}+\cdots+v_{\tau(i)})}\varepsilon(\tau)\cdot\\&&\rho_j\big(T\rho_i(Tv_{\tau(1)},\cdots,\hat{T}v_{\tau(s)},\cdots,Tv_{\tau(i)},v_{\tau(s)}), Tv_{\tau(i+1)},\cdots,Tv_{\tau(n-1)},v_{n}\big)\\
&&+\sum_{i+j=n+1\atop i\geq1,j\geq1}\sum_{\tau\in\mathbb S_{(j-1,i-1)}}(-1)^{v_{\tau(1)}+v_{\tau(2)}+\cdots+v_{\tau(j-1)}}\varepsilon(\tau)\rho_j\big(Tv_{\tau(1)},\cdots,Tv_{\tau(j-1)},\rho_i( Tv_{\tau(j)},\cdots,Tv_{\tau(n-1)},v_{n})\big)\\
&\stackrel{\eqref{rota-baxter-o-homotopy-lie}}{=}&\sum_{i+j=n+1\atop i\ge1,j\geq2}\sum_{\tau\in\mathbb S_{(i,j-2)}}\varepsilon(\tau)\rho_j\big(l_i(Tv_{\tau(1)},\cdots,Tv_{\tau(i)}), Tv_{\tau(i+1)},\cdots,Tv_{\tau(n-1)},v_{n}\big)\\
&&+\sum_{i+j=n+1\atop i\geq1,j\geq1}\sum_{\tau\in\mathbb S_{(j-1,i-1)}}(-1)^{v_{\tau(1)}+v_{\tau(2)}+\cdots+v_{\tau(j-1)}}\varepsilon(\tau)\rho_j\big(Tv_{\tau(1)},\cdots,Tv_{\tau(j-1)},\rho_i( Tv_{\tau(j)},\cdots,Tv_{\tau(n-1)},v_{n})\big)\\
&\stackrel{\eqref{sh-Lie-rep}}{=}&0.
\end{eqnarray*}
Thus,   $(V,\{\oprn_k\}_{k=1}^{+\infty})$ is a pre-Lie$_\infty$-algebra.
\end{proof}

\begin{cor}
With the above conditions, the linear map $T$ is a strict $L_\infty$-algebra homomorphism from the sub-adjacent $L_\infty$-algebra $V^C$ to the initial $L_\infty$-algebra $(\g,\{l_k\}_{k=1}^{+\infty})$.
\end{cor}
\begin{proof}
It follows from Theorem \ref{relative-o-to-homotopy-pre-lie-thm} and Corollary \ref{homotopy-pre-lie-to-lie-cor}.
\end{proof}
 At the end of this section, we give the necessary and sufficient conditions on an $L_\infty$-algebra admitting a compatible pre-Lie$_\infty$-algebra.

\begin{pro}\label{compatible homotopy-pre-Lie-and-rota-baxter}
Let $(\g,\{l_k\}_{k=1}^{+\infty})$ be an $L_\infty$-algebra. Then there exists a compatible pre-Lie$_\infty$-algebra if and only if there exists an invertible strict homotopy relative Rota-Baxter operator on $(\g,\{l_k\}_{k=1}^{+\infty})$.
\end{pro}

\begin{proof}
Let $T$ be an invertible strict homotopy relative Rota-Baxter operator on $(\g,\{l_k\}_{k=1}^{+\infty})$ with respect to a representation $(V,\{\rho_k\}_{k=1}^{+\infty})$. By Theorem \ref{relative-o-to-homotopy-pre-lie-thm}, $(V,\{\oprn_k\}_{k=1}^{+\infty})$ is a pre-Lie$_\infty$-algebra structure, where $\oprn_k$ is defined by \eqref{relative-o-to-homotopy-pre-lie}.
Since $T$ is an invertible linear map, there is an isomorphic pre-Lie$_\infty$-algebra structure $ \{\Oprn_k\}_{k=1}^{+\infty} $ on $\g$  given by
\begin{eqnarray}\label{eq:Tkinv}
\Oprn_k(x_1,\cdots,x_k):=T\oprn_k(T^{-1}x_1,\cdots,T^{-1}x_{k-1},T^{-1}x_k)=T\rho_k(x_1,\cdots,x_{k-1},T^{-1}x_k)
\end{eqnarray}
for all $x_1\cdots,x_k\in \g.$ Since $T$ is a strict homotopy relative Rota-Baxter operator, we have
\begin{eqnarray*}
l_k(x_1,\cdots,x_{k-1},x_{k})&=&\sum_{i=1}^{k}(-1)^{(x_{i+1}+\cdots+x_k)x_i}T\rho_{k}(x_1,\cdots,x_{i-1},x_{i+1},\cdots,x_k,T^{-1}x_i)\\
                             &=&\sum_{i=1}^{k}(-1)^{(x_{i+1}+\cdots+x_k)x_i}\Oprn_k(x_{1},\cdots,\hat{x}_{i},\cdots,x_{k},x_i).
\end{eqnarray*}
Therefore $(\g,\{\Oprn_k\}_{k=1}^{+\infty})$ is a compatible pre-Lie$_\infty$-algebra of $(\g,\{l_k\}_{k=1}^{+\infty})$.

Conversely, by Proposition \ref{homotopy-pre-lie-id-o}, the identity map $\Id$ is a strict homotopy relative Rota-Baxter operator on the sub-adjacent $L_\infty$-algebra $\g^C$ with respect to the representation $(\g,\{L_k\}_{k=1}^{+\infty})$.
\end{proof}

\noindent{\em Acknowledgements.} This research was partially supported by NSFC (11922110). R. Tang is also Funded by China Postdoctoral Science Foundation (2020M670833). This work was completed in part while the first author was visiting Max Planck Institute for Mathematics in Bonn and he wishes to thank this institution for excellent working conditions.


\begin{thebibliography}{a}

 \bibitem{Arnal}
 D. Arnal,  Simultaneous deformations of a Lie algebra and its modules. Differential geometry and mathematical physics (Liege, 1980/Leuven, 1981), 3-15, \emph{Math. Phys. Stud.}, 3, Reidel, Dordrecht, 1983.





\bibitem{Bai}
C. Bai, A unified algebraic approach to the classical Yang-Baxter equation. \emph{J. Phys. A: Math. Theor.} {\bf 40} (2007), 11073-11082.

\bibitem{BBGN} C. Bai, O. Bellier, L. Guo and X. Ni, Spliting of
operations, Manin products and Rota-Baxter operators.  {\em Int. Math. Res. Not.} {\bf 3} (2013), 485-524.

\bibitem{Bal} D. Balavoine, Deformations of algebras over a quadratic operad. Operads: Proc. of Renaissance Conferences (Hartford, CT/Luminy, 1995), \emph{Contemp. Math.} {\em 202} Amer. Math. Soc., Providence, RI, 1997, 207-34.

\bibitem{Barmeier}
S. Barmeier and Y. Fr\'egier, Deformation-obstruction theory for diagrams of algebras and applications to geometry. to appear in {\em J. Noncommut. Geom.} arXiv:1806.05142.


\bibitem{Ba} G. Baxter, An analytic problem whose solution follows from a simple algebraic identity. \emph{Pacific J. Math.} {\bf 10} (1960), 731-742.


\bibitem{Bor}
M. Bordemann, Generalized Lax pairs, the modified classical Yang-Baxter equation, and affine geometry
of Lie groups. \emph{Comm. Math. Phys.} {\bf 135} (1990), 201-216.

\bibitem{Borisov}
D. V. Borisov, Formal deformations of morphisms of associative algebras. {\em Int. Math. Res. Not.} {\bf 41} (2005), 2499-2523.


\bibitem{CL}
F. Chapoton and M. Livernet, Pre-Lie algebras and the rooted trees operad. {\em  Int. Math. Res. Not.}  {\bf 8} (2001), 395-408.

\bibitem{CP}
V. Chari and A. Pressley, A Guide to Quantum Groups. Cambridge University Press, 1994.

\bibitem{Ch-Ei}
C. Chevalley and S. Eilenberg,
\newblock Cohomology theory of {L}ie groups and {L}ie algebras.
\newblock {\em Trans. Amer. Math. Soc.} {\bf 63} (1948), 85-124.

\bibitem{CK}
A. Connes and D. Kreimer, { Renormalization in quantum field theory and the Riemann-Hilbert problem. I. The Hopf algebra structure of graphs and the main theorem.} {\em Comm. Math. Phys.} {\bf 210} (2000), 249-273.


\bibitem{Dolgushev-Rogers}
V. A. Dolgushev and C. L. Rogers, A version of the Goldman-Millson Theorem
for filtered $L_\infty$-algebras. {\em J. Algebra}  {\bf 430} (2015), 260-302.

\bibitem{Doubek-Markl-Zima}
M. Doubek, M. Markl and P. Zima, Deformation theory (lecture notes). \emph{ Arch. Math. (Brno)} {\bf 43} (2007),  333-371.

\bibitem{DK}
V. Dotsenko and A. Khoroshkin, Quillen homology for operads via Gr\"{o}bner bases. \emph{Doc. Math.} {\bf 18} (2013), 707-747.


\bibitem{Fard}
K Ebrahimi-Fard, D. Manchon and F. Patras, A noncommutative Bohnenblust-Spitzer identity for Rota-Baxter algebras solves Bogoliubov's counterterm recursion.  {\em J. Noncommut. Geom.} {\bf 3} (2009), 181-222.


\bibitem{Fregier}
Y. Fr\'egier, M. Markl and D. Yau,   The
$L_\infty$-deformation complex of diagrams of algebras. \emph{New York J. Math.} {\bf 15} (2009), 353-392.

\bibitem{Fregier-Zambon-1}
Y. Fr\'egier, and M. Zambon, Simultaneous deformations and Poisson geometry. \emph{ Compos. Math.} {\bf 151} (2015), 1763-1790.

\bibitem{Fregier-Zambon-2}
Y. Fr\'egier, and M. Zambon, Simultaneous deformations of algebras and morphisms via derived brackets. \emph{J. Pure Appl. Algebra} {\bf 219 } (2015), 5344-5362.



\bibitem{Ge0}
M. Gerstenhaber, The cohomology structure of an associative ring. \emph{Ann. Math.} {\bf 78} (1963), 267-288.

\bibitem{Ge}
M. Gerstenhaber, On the deformation of rings and algebras. \emph{Ann. Math. (2) } {\bf 79} (1964), 59-103.

\bibitem{Getzler}
E. Getzler, Lie theory for nilpotent $L_{\infty}$-algebras. {\em Ann.   Math. (2)} {\bf 170} (2009), 271-301.

\bibitem{Goncharov}
M. E. Goncharov and P. S. Kolesnikov, Simple finite-dimensional double algebras. {\em J. Algebra}  {\bf 500} (2018), 425-438.

\bibitem{GLST} A. Guan, A. Lazarev, Y. Sheng, and R. Tang, Review of deformation theory II: a homotopical approach.  	\emph{Adv. Math. (China)} {\bf 49} (2020), 278-298.

\bibitem{Gub-AMS}
L. Guo,  What is a Rota-Baxter algebra? {\em Notices of the AMS} {\bf 56}  (2009), 1436-1437.



\bibitem{Gub}
L. Guo,  An introduction to Rota-Baxter algebra. Surveys of Modern Mathematics, 4. International Press, Somerville, MA; Higher Education Press, Beijing, 2012. xii+226 pp.

\bibitem{Gu0-1}
L. Guo, Properties of Free Baxter Algebras. \emph{Adv. Math.}  {\bf 151}  (2000), 346-374.


\bibitem{Hamilton-Lazarev}
A. Hamilton and A. Lazarev, Cohomology theories for homotopy algebras and noncommutative geometry. \emph{Algebr. Geom. Topol.} {\bf 9} (2009), 1503-1583.

\bibitem{Har}
D.~K. Harrison,
\newblock Commutative algebras and cohomology.
\newblock {\em Trans. Amer. Math. Soc.}  {\bf 104} (1962), 191-204.


\bibitem{Hor}
G.~Hochschild,
\newblock On the cohomology groups of an associative algebra.
\newblock {\em Ann.  Math. (2)}  {\bf 46} (1945), 58-67.



\bibitem{KSo}
M. Kontsevich and Y. Soibelman, Deformation theory. I [Draft], http://www.math.ksu.edu/~soibel/Book-vol1.ps, 2010.

\bibitem{Kosmann-Schwarzbach}
 Y. Kosmann-Schwarzbach, From Poisson algebras to Gerstenhaber algebras. {\em Ann. Inst. Fourier (Grenoble) } {\bf 46} (1996), 1243-1274.


\bibitem{Ku}
B. A. Kupershmidt, What a classical $r$-matrix really is. \emph{J.
Nonlinear Math. Phys.} {\bf 6} (1999), 448-488.

\bibitem{LS}
T. Lada and J. Stasheff, Introduction to sh Lie algebras for
physicists. \emph{Internat. J. Theoret. Phys.} {\bf 32} (1993), 1087-1103.

\bibitem{LM}
T. Lada and M. Markl,  Strongly homotopy Lie algebras. \emph{ Comm. Algebra} {\bf 23} (1995),  2147-2161.



\bibitem{Lazarev-1}
A. Lazarev, Y. Sheng and R. Tang, Homotopy Rota-Baxter algebras, triangular $L_\infty$-bialgebras and higher derived brackets. \texttt{arXiv:2008.00059.}


\bibitem{LV}
J.-L. Loday and B. Vallette, Algebraic Operads. Springer, 2012.

\bibitem{Lu} J. Lurie, \emph{ DAG X: Formal moduli problems}, available at http://www.math.harvard.edu/~lurie/papers/DAG-X.pdf

\bibitem{Ma-0}
M. Markl, Intrinsic brackets and the $L_{\infty}$-deformation theory of bialgebras. \emph{J. Homotopy Relat. Struct.} {\bf5} (2010), 177-212.


\bibitem{Ma}
M. Markl, Deformation Theory of Algebras and Their Diagrams. \emph{Regional Conference Series in Mathematics}, Number 116, American Mathematical Society (2011).

\bibitem{MSS} M. Markl, S. Shnider and J. D. Stasheff, Operads in
    Algebra, Topology and Physics. American Mathematical Society, Providence, RI, 2002.

\bibitem{Mer}
S. A. Merkulov, Nijenhuis infinity and contractible differential graded manifolds. \emph{ Compos. Math.} {\bf 141} (2005),  1238-1254.




\bibitem{NR} A. Nijenhuis  and R. Richardson,  Cohomology and deformations in graded Lie algebras. {\em Bull.
Amer. Math. Soc.} {\bf 72} (1966), 1-29.

\bibitem{NR2} A. Nijenhuis and R. Richardson,  Commutative  algebra cohomology and deformations of Lie and associative algebras. {\em J. Algebra} {\bf 9} (1968), 42-105.

\bibitem{PBG} J. Pei, C. Bai and L. Guo, Splitting of Operads and
Rota-Baxter Operators on Operads. \emph{Appl. Categor. Struct.}
{\bf 25} (2017), 505-538.

\bibitem{Pr} J. P. Pridham, Unifying derived deformation theories. \emph{Adv. Math.} {\bf 224} (2010), 772-826.


\bibitem{STS} M.~A. Semyonov-Tian-Shansky, What is a
classical R-matrix? \emph{Funct. Anal. Appl.} {\bf 17} (1983), 259-272.


\bibitem{Sheng}
Y. Sheng, Categorification of pre-Lie algebras and solutions of $2$-graded classical Yang-Baxter equations. \emph{Theory Appl.  Categ.} {\bf 34} (2019), 269-294.

\bibitem{Sta63} J. Stasheff, Homotopy associativity of H-spaces. I, II. \emph{Trans. Amer. Math. Soc.} {\bf 108} (1963), 275-292; ibid. {\bf 108} (1963), 293-312.

\bibitem{Stasheff}
J. Stasheff, The intrinsic bracket on the deformation complex of an associative algebra. \emph{J. Pure Appl. Algebra}  {\bf89} (1993), 231-235.

\bibitem{stasheff:shla} J. Stasheff,  Differential graded {L}ie algebras, quasi-Hopf algebras and higher homotopy algebras.   \emph{Quantum groups (Leningrad, 1990)}, 120-137, Lecture Notes in Math., 1510, \emph{ Springer, Berlin,} 1992.

\bibitem{St19} J. Stasheff, $L$-infinity and $A$-infinity structures. {\em   High. Struct.} {\bf 3} (2019), 292-326.


\bibitem{TBGS-1}
R. Tang, C. Bai, L. Guo and Y. Sheng, Deformations and their controlling cohomologies of $\huaO$-operators. \emph{Comm. Math. Phys.} {\bf 368} (2019), 665-700.



\bibitem{Vo}
Th. Voronov, Higher derived brackets and homotopy algebras. \emph{J. Pure Appl. Algebra}  {\bf 202} (2005), 133-153.

\bibitem{Yu-Guo}
H. Yu, L. Guo and J.-Y. Thibon, Weak quasi-symmetric functions, Rota-Baxter algebras and Hopf algebras. \emph{Adv. Math.} {\bf 344} (2019), 1-34.



\end{thebibliography}
\end{document}